\documentclass[12pt]{article}

\usepackage[francais]{babel}
\usepackage[utf8]{inputenc}
\usepackage[T1]{fontenc}
\usepackage[top=2.5cm,bottom=2.5cm,left=2.5cm,right=2.5cm]{geometry}

\usepackage{amsmath, amssymb, amsthm}
\usepackage{graphicx, float}
\graphicspath{{Images/}}
\usepackage{mathrsfs} 
\usepackage{url, hyperref, xcolor}

\newtheorem*{thm}{Th\'eor\`eme}
\newtheorem*{prop}{Proposition}

\newtheorem*{rem}{Remarque}
\newtheorem*{quest}{Question}
\newtheorem*{conj}{Conjecture}
\newtheorem*{ex}{Exemple}
\newtheorem{Define}{D\'efinition}[section]
\newtheorem{Thm}[Define]{Th\'eor\`eme}
\newtheorem{Lem}[Define]{Lemme}
\newtheorem{Cor}[Define]{Corollaire}
\newtheorem{Prop}[Define]{Proposition}

\newcommand{\C}{\mathbb{C}}
\renewcommand{\S}{\mathbb{S}}
\renewcommand{\P}{\mathbb{P}}
\newcommand{\R}{\mathbb{R}}
\newcommand{\N}{\mathbb{N}}
\newcommand{\Z}{\mathbb{Z}}

\newcommand{\D}{\mathbb{D}}
\newcommand{\Cc}{\mathcal{C}}
\newcommand{\CC}{\mathscr{C}}

\title{Topologie et dénombrement des \\ courbes algébriques réelles singulières}
\date{Automne 2019}
\author{Christopher-Lloyd Simon
\\ christopher-lloyd.simon@ens-lyon.fr
}

\begin{document}
\maketitle

\begin{abstract}
Nous décrivons la topologie des courbes algébriques réelles singulières dans une surface lisse.
Nous énumérons et bornons en fonction du degré le nombre de types topologiques de courbes algébriques singulières du plan projectif réel.


%
%
%
%


\paragraph{Mots clés :} Courbe algébrique réelle, éclatement, classe combinatoire, décomposition de Cunningham d'un graphe, grammaire algébrique, opérade, combinatoire analytique.


\end{abstract}

\setcounter{tocdepth}{2}
\tableofcontents

\section*{Introduction}

Nous nous intéressons dans un premier temps aux courbes analytiques singulières sur une surface analytique réelle. Nous décrivons puis énumérons leurs configurations topologiques, d'abord au voisinage des singularités, puis dans leur globalité.

Nous adaptons dans un second temps la description topologique aux courbes algébriques singulières sur une surface algébrique réelle lisse. Enfin, nous majorons le nombre de types topologiques de courbes algébriques singulières du plan projectif en fonction du degré.

Les caractérisations topologiques sont formulées en termes d'objets combinatoires comme des diagrammes de cordes ou des graphes, et les résultats d'énumération portent sur la nature des séries génératrices, les premiers termes, et l'asymptotique des suites de nombres énumérant les objets combinatoires en question.

\subsection*{Topologie et combinatoire des singularités analytiques}

Au voisinage d'une singularité, une courbe analytique est une union de branches qui traversent le point singulier. Traçons un petit cercle centré sur la singularité : chacun de ces arcs rentre à l'intérieur, passe par la singularité et en ressort; intersectant ainsi le cercle en exactement deux points.

La structure locale d'une singularité analytique du plan est donc encodée par un diagramme de cordes : des points distincts sur le cercle appariés deux à deux, ou encore un mot cyclique dont chaque lettre apparaît deux fois. Les diagrammes ainsi obtenus, qualifiés d'analytiques, forment une infime proportion de tous les diagrammes de cordes. La première section en rappelle plusieurs caractérisations montrées dans \cite{GhySim:2020}.

Dans la deuxième section, nous approfondissons une description supplémentaire de ces diagrammes et nous l'exploitons afin de les énumérer par des méthodes de combinatoire analytique. Voici l'un des trois principaux résultats d'énumération des configurations locales.
\begin{thm}
La série génératrice ordinaire $A$ comptant le nombre $A_n$ de diagrammes analytiques enracinés est algébrique:
$(z^3+z^2)A^6-z^2A^5-4zA^4+(8z+2)A^3-(4z+6)A^2+6A-2=0$.
Ainsi $A_n\sim a_0\,n^{-\frac{3}{2}}\alpha^{-n}$ où $4<10^3a_0<5$ et $15.792395< \alpha^{-1}< 15.792396$.
Les premiers termes de la suite $A_n$ sont:
$1,1,3,15,105,923,9417,105815,1267681,15875631,205301361$.
\end{thm}

\subsection*{Topologie globale et dénombrement des courbes analytiques}

Globalement, une courbe analytique réelle sur une surface analytique compacte connexe est formée de parties lisses rejoignant ses singularités.
Dans la troisième section nous construisons d'abord un modèle pour cette structure appelé courbe combinatoire : c'est un graphe plongé dans la surface dont les sommets, associés aux singularités, sont décorés par leurs diagrammes de cordes, et dont les arêtes correspondent aux arcs lisses entre les singularités.

Ensuite, nous montrons qu'il n'y a aucune obstruction globale restreignant la topologie des courbes analytiques.
Autrement dit, dans une surface analytique quelconque, tout plongement d'un graphe de diagrammes de cordes analytiques provient d'une courbe analytique.
La preuve, qui utilise des résultats de Grauert et Cartan sur la représentabilité des variétés analytiques, se fait en deux étapes : on résout les singularités par éclatement puis on approche à la Whitney les courbes lisses par des courbes analytiques.

Cette description nous permet enfin d'associer le dénombrement des diagrammes de cordes analytiques avec celui des découpages de Tutte \cite{Tutte:1962} pour énumérer, en fonction du nombre d'arêtes, les courbes combinatoires de la sphère provenant de courbes analytiques singulières.

\begin{prop}
Le nombre de courbes combinatoires analytiques de la sphère ayant $c$ arêtes pour $s$ sommets indexés et marqués de tailles respectives $k_1,\dots, k_v$ vaut:
\[
\frac{(c-1)!}{(c-s-2)!} \: \prod_{v=1}^{s}{k_v \binom{2k_v}{k_v}A_{k_v}}
.\]
\end{prop}

\subsection*{Topologie globale et dénombrement des courbes algébriques}

Dans la section 4, nous expliquons d'abord quelques notions propres aux courbes et aux surfaces algébriques en vue d'énnoncer le théorème d'approximation de Nash-Tognoli des courbes lisses par des courbes algébriques.

Nous déterminons, en éclatant les singularités et en approchant les courbes, quelles obstructions globales restreignent la topologie de ces courbes algébriques.
Ces obstructions de nature homologique dépendent de la structure algébrique de la surface sous-jacente, et disparaissent pour une surface rationnelle : sur une sphère ou un plan projectif, toute courbe combinatoire dont les sommets sont décorés par des diagrammes de cordes analytiques provient d'une courbe algébrique.

Nous établissons enfin, en associant la formule précédente avec les formules de Pl\"ucker, une majoration du nombre de types topologiques de courbes algébriques réelles de degré $d$ du plan projectif dont le lieu réel est connexe.

\begin{thm}
Le nombre $Cal_{\,\R\P^2\,}(d)$ de types toplogiques de courbes algébriques réelles connexes de degré $d$ du plan projectif se comporte comme:  
\[Cal_{\,\R\P^2\,}(d)=o\left(12^{d^2}\right)\]
\end{thm}

\paragraph{Remerciements.}
Cet article est le fruit de mes recherches effectuées sous la direction d'\'Etienne Ghys durant l'été 2018 ; il prend racines dans les problématiques exposées au long de sa promenade \cite{Ghys:2017}. Sa disponibilité, son encouragement et son exigence me donnent l'envie d'aller au fond des choses et m'ont permis d'amener ce travail à maturité.

Les discussions combinatoires avec Mickaël Maazoun m'ont guidé et débloqué à maintes reprises; je tiens à le remercier chaleureusement pour son temps et sa bonne humeur.
Je suis également redevable à David Coulette pour avoir optimisé mon implémentation symbolique de l'inversion de Lagrange, et tracé proprement les courbes algébriques de ce document.
Je remercie enfin Patrick Popescu-Pampu pour ses relectures attentives ainsi qu'Olivier Benoist pour m'avoir expliqué le théorème de Nash-Tognoli.


\section{Notions préliminaires}

Résumons d'abord la description topologique d'une courbe analytique réelle $\{f(x,y)=0\}$ du plan $\R^2$, au voisinage d'une singularité. Elle est formulée dans \cite{GhySim:2020} au moyen du diagramme de cordes associé et du graphe d'entrelacement de ses branches (voir aussi \cite{Ghys:2017}). Nous esquisserons ensuite la factorisation de Cunningham, utile dans la deuxième section pour approfondir l'étude combinatoire de ces deux familles d'invariants: graphes d'entrelacement et diagrammes de cordes.

\subsection{Topologie locale : graphe d'entrelacement des branches}

\paragraph{De la courbe au diagramme de cordes.}
Soit $\gamma$ un germe de courbe analytique réelle du plan orienté au point $o$ qui n'est pas vide ou réduit à un point; et $\D_\epsilon$ le disque de rayon $\epsilon$ centré en $o$.
Pour $\epsilon$ suffisamment petit tous les $\gamma \cap \D_\epsilon$ sont homéomorphes donc la topologie de $\gamma$ sur un voisinage de l'origine est définie sans ambiguïté.
Plus précisément, la courbe $(\gamma,o)$ est localement homéomorphe au cône sur son intersection avec un petit cercle $\gamma \cap \partial \D_\epsilon$ appelée son \emph{halo réel}.
Par ailleurs, la courbe $\gamma$ est l'union d'un nombre fini de branches, chaque branche donnant lieu à exactement deux points du halo.
Par conséquent, ce dernier définit un \emph{diagramme de cordes}: un nombre pair $2c$ de points distincts du cercle orienté, deux à deux appariés par une corde comme sur la figure \ref{fig:cordiag}. Formellement, c'est un mot cyclique sur un alphabet à $c$ lettres dont chaque lettre apparaît deux fois.
Les diagrammes de cordes provenant de germes de courbes analytiques seront dit \emph{analytiques}.

\begin{figure}[H]
 \centering
 \includegraphics[width=0.4\textwidth]{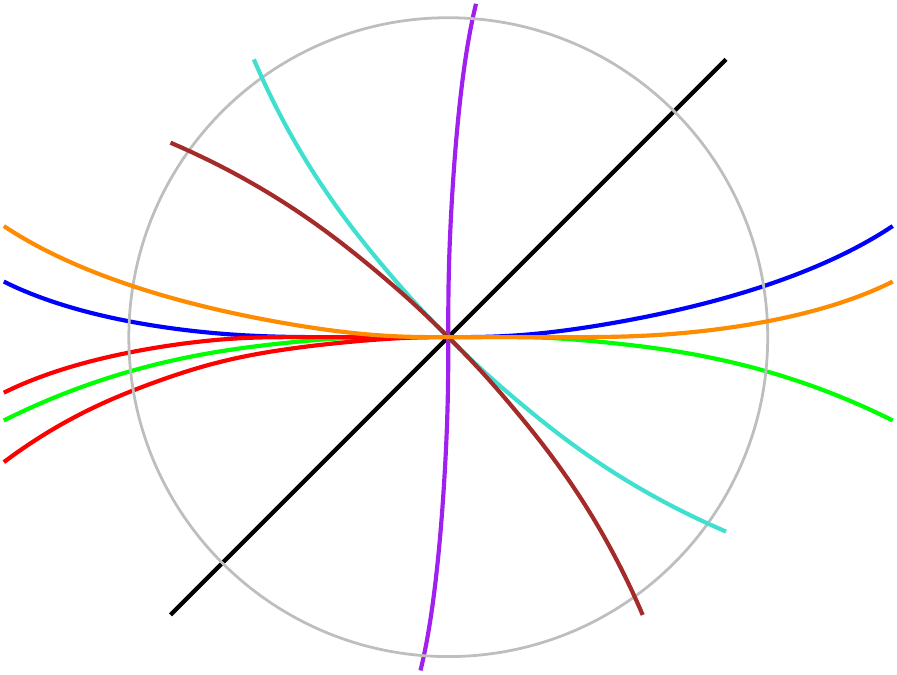}
 \includegraphics[width=0.38\textwidth]{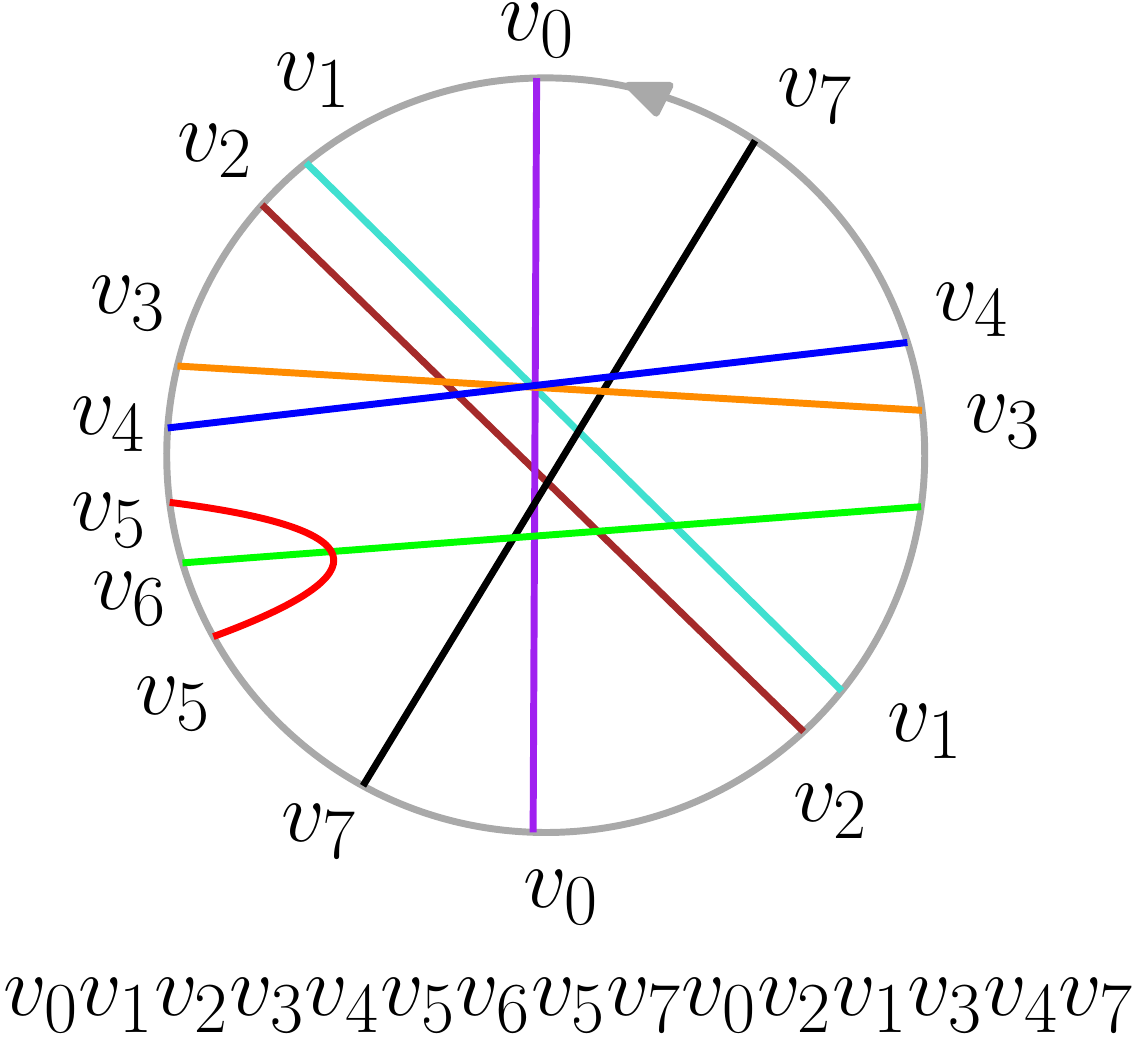}
 \caption{\label{fig:cordiag} Représentation d'une singularité analytique par un diagramme de cordes.}
\end{figure}

L'article \cite{GhySim:2020} donne un moyen récursif pour déterminer l'analycité d'un diagramme. Supposons qu'un diagramme de cordes $c$ possède l'un des motifs suivants:
\begin{itemize}
    \item une corde isolée: $\dots aa \dots $
    \item une fourche: $\dots b \dots aba $
    \item une paire de vrais jumeaux: $\dots ab \dots ab$
    \item une paire de faux jumeaux: $\dots ab \dots ba$
\end{itemize}
On appelle \emph{simplification} le diagramme $c'$ obtenu en enlevant la corde $a$ d'un tel motif.

\begin{figure}[H]
 \centering
 \includegraphics[width=0.6\textwidth]{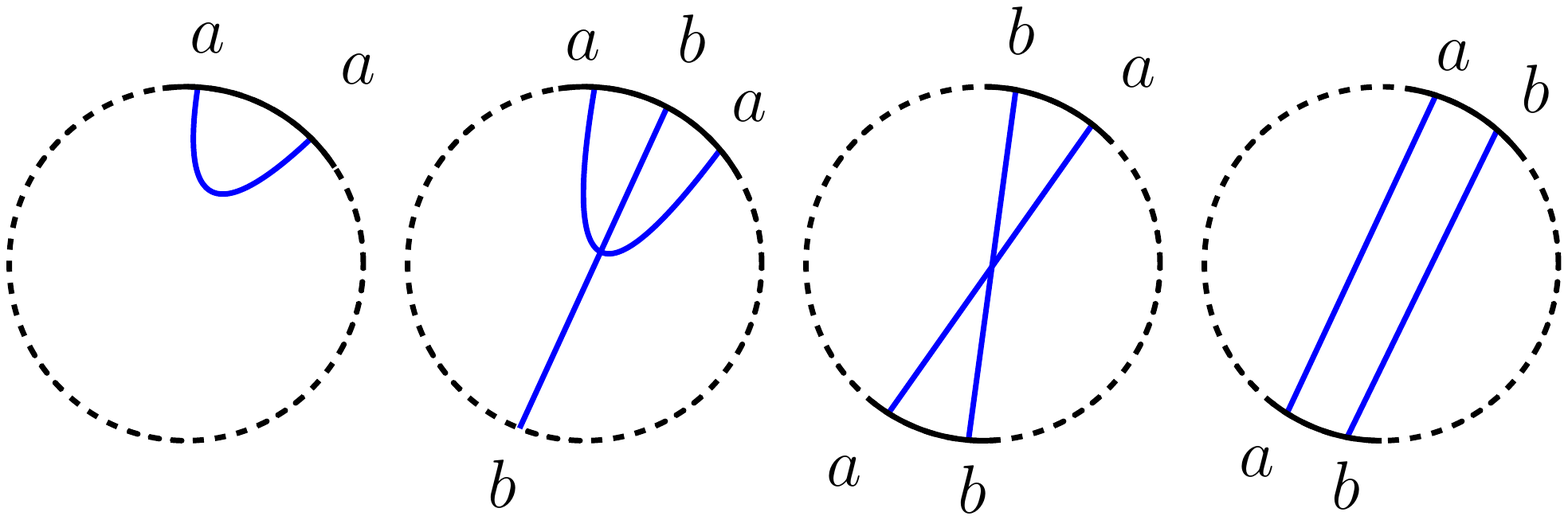}
 \caption{\label{fig:isole_fourche_jumeaux} Corde isolée, fourche, vrais et faux jumeaux dans le diagramme.}
\end{figure}

\begin{Thm}[Description récursive des diagrammes analytiques \cite{Ghys:2017}]
Le mot vide est un diagramme analytique.
Si un diagramme non vide ne contient pas de corde isolée, de fourche, ou une paire de jumeaux; alors il n'est pas analytique.
Si par contre il contient un tel motif, alors il est analytique si et seulement si une de ses simplifications quelconques l'est aussi.
\end{Thm}

\begin{rem}(Analytique versus algébrique)
Les diagrammes de cordes provenant des singularités algébriques et analytiques sont les mêmes.
\end{rem}

\paragraph{Du diagramme de cordes au graphe d'entrelacement.}

Dans cette section, un \emph{graphe} sera la donnée d'un ensemble $V$ de sommets et d'un sous-ensemble d'arêtes $E$ des parties à deux éléments de $V$. Il n'y a donc ni boucles ni arêtes multiples. L'ensemble $V$ est muni de la \emph{métrique de graphe} $d_G(v,v')$ à valeurs dans $\N\cup\{\infty \}$ qui calcule l'infimum du nombre d'arêtes parmi les chemins entre deux sommets $v$ et $v'$. Notons $N_G(v)$ l'ensemble des voisins de $v$, définis comme les sommets à distance un de $v$.

Le \emph{graphe d'entrelacement} d'un diagramme de cordes (associé ou non à un germe de courbe) est construit en choisissant un sommet pour chaque corde, et en reliant deux sommets distincts lorsque les cordes associées s'intersectent, c'est-à-dire que leurs extrémités s'entrelacent pour l'ordre cyclique. Le graphe d'entrelacement du diagramme de cordes à droite de la figure \ref{fig:cordiag} est représenté à droite de la figure \ref{fig:acessibility}.

\begin{rem}[Graphes d'entrelacements de diagrammes de cordes]
Deux graphes sont localement équivalents si l'on peut passer de l'un à l'autre par une suite de transformations locales, consistant à inverser la relation d'incidence autour d'un sommet donné.

Bouchet a montré dans \cite{Bouchet:1994} que les graphes d'entrelacement de diagrammes de cordes sont précisément ceux dont aucun équivalent local ne contient l'un des trois motifs suivants comme sous-graphe induit.
\end{rem}

\begin{figure}[H]
 \centering
 \includegraphics[width=0.7\textwidth]{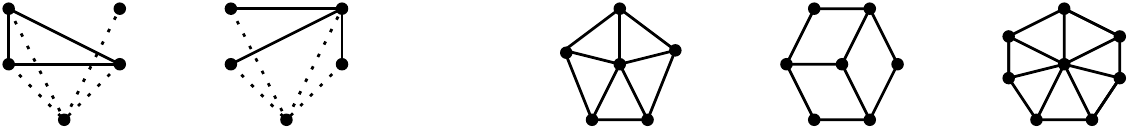}
 \caption{\label{fig:bouchet_local_motifs} Transformation locale renversant l'incidence au sommet central. Motifs interdits. }
\end{figure}

L'analyticité d'un diagramme de cordes se lit sur son graphe d'entrelacement, et le théorème suivant donne diverses caractérisations de ces graphes. Pour une discussion plus exhaustive des liens entre ces propriétés on pourra consulter l'article \cite{GhySim:2020} et ses références.

\begin{Thm}[Diagrammes de cordes analytiques \cite{GhySim:2020}]
Les graphes d'entrelacement $G=(V,E)$ provenant d'un diagramme de cordes analytique sont caractérisés par chacune des propriétés équivalentes suivantes:
\begin{itemize}
    \item \emph{Repliable}: se réduit à un ensemble de sommets isolés par suppression itérée:
    \begin{itemize}
        \item de sommets pendants: incidents à une seule arête,
        \item de sommets $v$ ayant un jumeau $v'$: c'est-à-dire qui a les mêmes voisins pourvu que l'on ignore toute arête éventuelle entre eux.
    \end{itemize}
    \item \emph{Distance-héréditaire}: pour tout sous-graphe induit $G'$ par un ensemble $V'\subset V$, la métrique induite par $d_G$ sur $V'$ est égale à $d_{G'}$.  
    \item \emph{Buisson}: pour quatre sommets $a_i,a_j,a_k,a_l$ quelconques, deux des trois sommes des longueurs de paires de diagonales opposées: $d_G(a_i,a_j)+d_G(a_k,a_l)$ sont égales.
    \item Ne contient pas de maison, gemme, domino ou $(n\geq5)$-cycle comme graphe induit.
\end{itemize}  
\end{Thm}

\begin{figure}[H]
 \centering
 \includegraphics[width=0.5\textwidth]{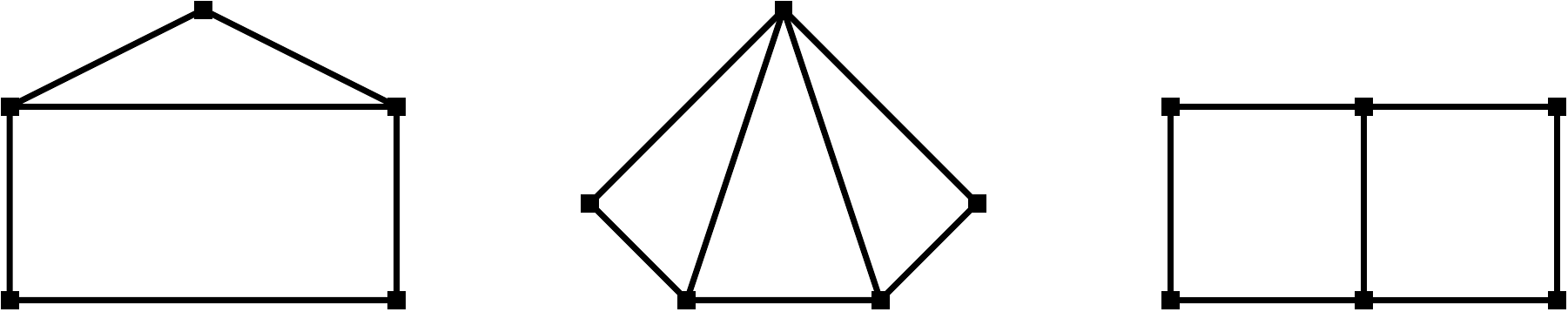}
 \caption{\label{fig:house-gem-domino} Maison, Gemme, Domino.}
\end{figure}

\begin{rem}[terminologique]
Le nom buisson semble approprié pour décrire la troisième propriété ainsi que pour suggérer l'allure des graphes qu'elle définit ; c'est désormais ainsi que nous les désignerons.

En effet, la définition métrique ci-dessus est similaire à celle des espaces $0$-hyperboliques au sens de Gromov. En effet, pour ces derniers on requiert l'égalité des deux plus grandes parmi les trois quantités $d_G(a_i,a_j)+d_G(a_k,a_l)$. Or toute partie finie d'un espace $0$-hyperbolique est isométrique à celle d'un arbre métrique et réciproquement.

Par ailleurs l'apparence géométrique que peuvent prendre ces graphes fait penser à des buissons: certaines parties sont arborescentes, tandis que d'autres sont plus touffues. Comme nous le verrons plus loin, cette structure est plus apparente après leur avoir appliqué la décomposition en arbre-de-graphes.

\end{rem}

Les seuls diagrammes de cordes dont le graphe d'entrelacement est une maison, une gemme, un domino et un $n>4$-cycle, sont ceux de la figure suivante.

\begin{figure}[H]
 \centering
 \includegraphics[width=0.5\textwidth]{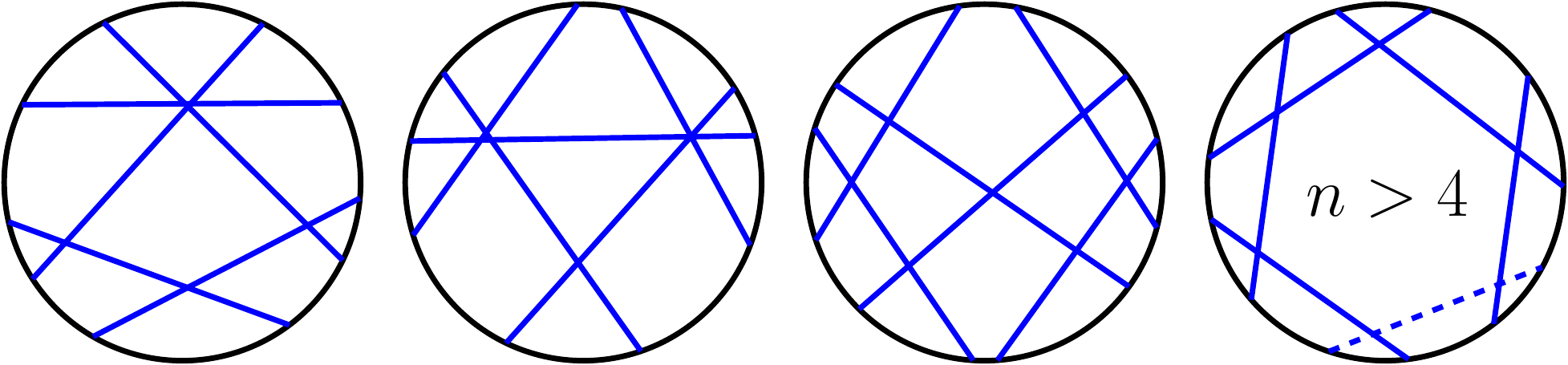}
 \caption{\label{fig:cordiag_motifs_interdits} Diagrammes dont l'entrelacement est une maison, gemme, domino, $n>4$ cycle.}
\end{figure}

\begin{Cor}[Analycité: motifs interdits]
Un diagramme de cordes est analytique si, et seulement si, il ne contient pas comme sous-diagramme l'un de ceux de la figure \ref{fig:cordiag_motifs_interdits}.
\end{Cor}

\subsection{La décomposition de Cunningham}

Rappelons brièvement le théorème de décomposition de Cunningham revisité par Gioan et Paul \cite{GioPau:2012} (voir aussi \cite[section 4.8.5]{ChDuMo:2012}). Dans ce paragraphe nous ne considérons que des graphes connexes.

Une \emph{décomposition} d'un graphe $G$ est une partition de l'ensemble de ses sommets en deux ensembles $A_1$ et $A_2$ ayant au moins deux éléments chacun, et contenant tous les deux des ensembles de sommets $B_j\subset A_j$ tels que les arêtes entre les  $A_j$ soient exactement les arêtes du graphe biparti complet sur les $B_j$.

\begin{figure}[H]
 \centering
 \includegraphics[width=0.45\textwidth]{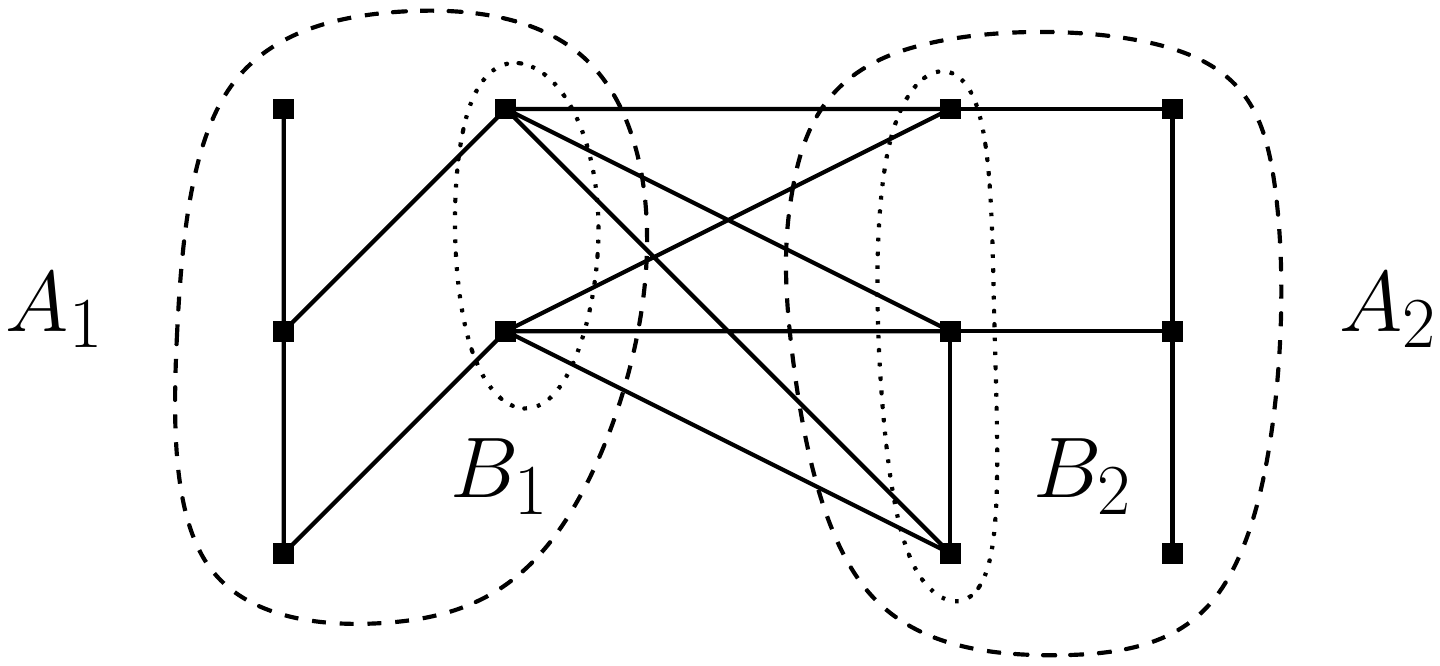}
 \caption{\label{fig:split} Un graphe décomposable.}
\end{figure}

Une telle décomposition permet de factoriser $G$ en séparant les graphes induits par les ensembles $A_j$, que l'on assemble dans une même structure en ajoutant une arête $\{x_1,x_2\}$ d'un autre type, dont chaque extrémité $x_j$ est connectée aux sommets de $B_j$ (figure \ref{fig:split_facto} gauche).
En ajoutant une feuille connectée à chaque sommet de $G$ (figure \ref{fig:split_facto} milieu), cela donne un \emph{arbre-de-graphes}; c'est-à-dire un arbre $T$ auquel est associé à chaque n\oe ud interne $k$ un graphe connexe $G_k$ muni d'une bijection $\varphi_k$ de l'ensemble de ses sommets vers les sommets $N_T(k)$ de l'arbre qui sont adjacents au n\oe ud $k$.
Après la première factorisation, l'arbre possède deux n\oe uds internes reliés par une arête et décorés par les graphes $G_k$ induits par les $A_k\cup \{x_k\}$.
On peut ensuite poursuivre la factorisation en essayant de décomposer les $G_k$, et ainsi de suite, pour obtenir une série d'arbres-de-graphes. La factorisation se termine lorsque les n\oe uds sont indécomposables ou dégénérés; expliquons ces termes.

\begin{figure}[H]
 \centering
 \includegraphics[width=1\textwidth]{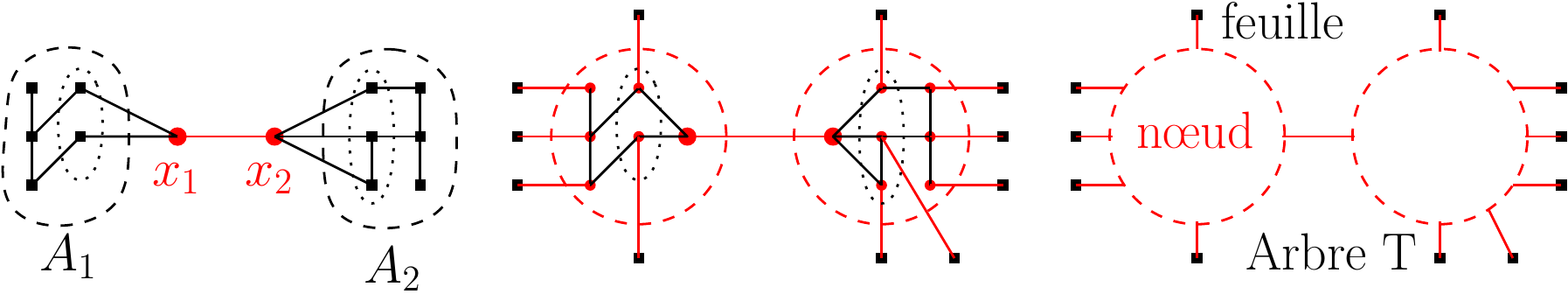}
 \caption{\label{fig:split_facto} Factorisation en arbre-de-graphes avec $2$ n\oe uds $G_k$ reliés par une arête.}
\end{figure}

Un graphe est \emph{indécomposable} s'il a au moins quatre sommets et n'admet aucune décomposition. Par exemple la maison, la gemme, le domino et les $(n>4)$-cycles sont indécomposables.
A l'opposé, un graphe est \emph{dégénéré} si toute bipartition de ses sommets en parties de cardinal au moins deux est une décomposition: un tel graphe, s'il est connexe, est soit une clique $K_n$ soit une étoile $S_n$ (star) c'est-à-dire le graphe biparti $K_{1,n}$. En particulier tout graphe sur au plus trois sommets est dégénéré.

\begin{figure}[H]
 \includegraphics[width=0.5\textwidth]{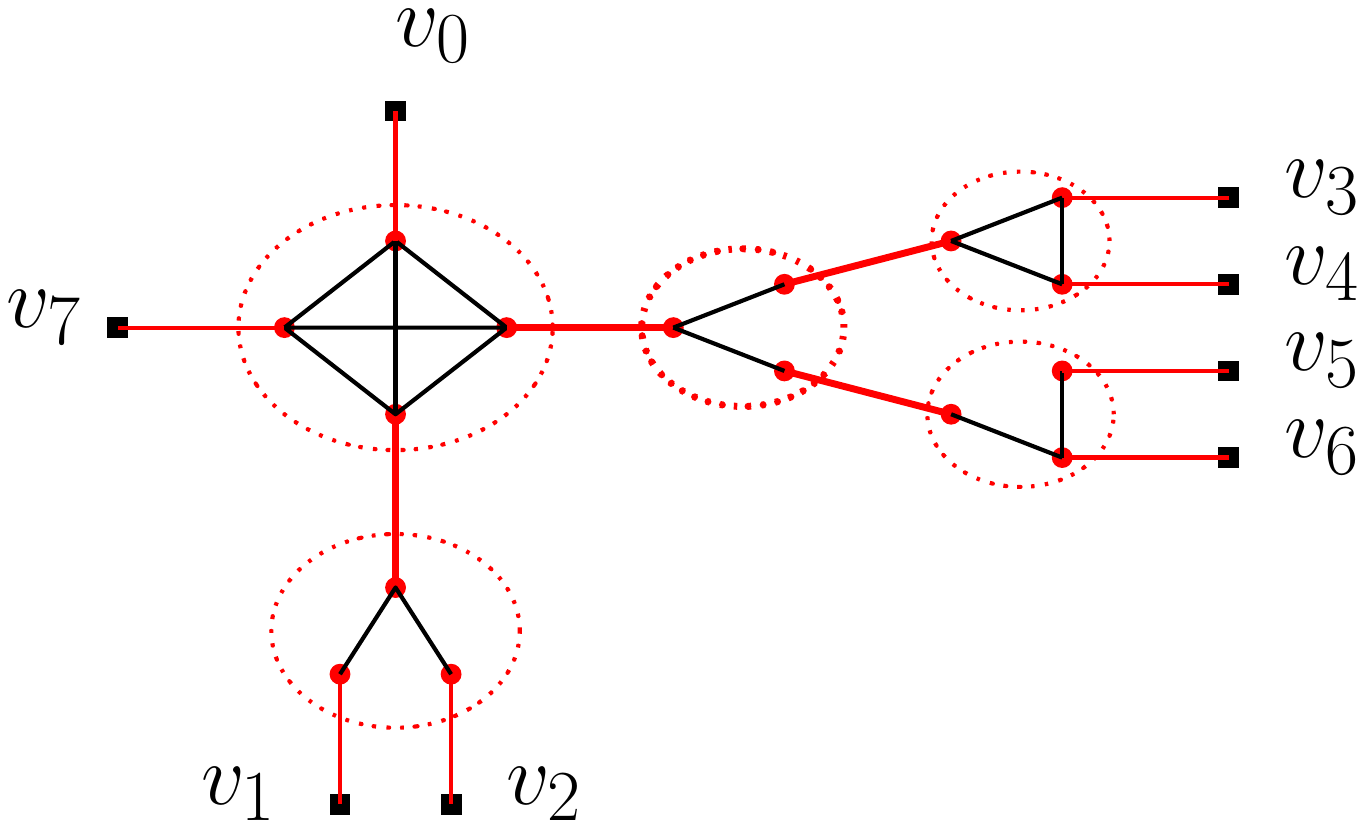}
 \hfill
 \includegraphics[width=0.4\textwidth]{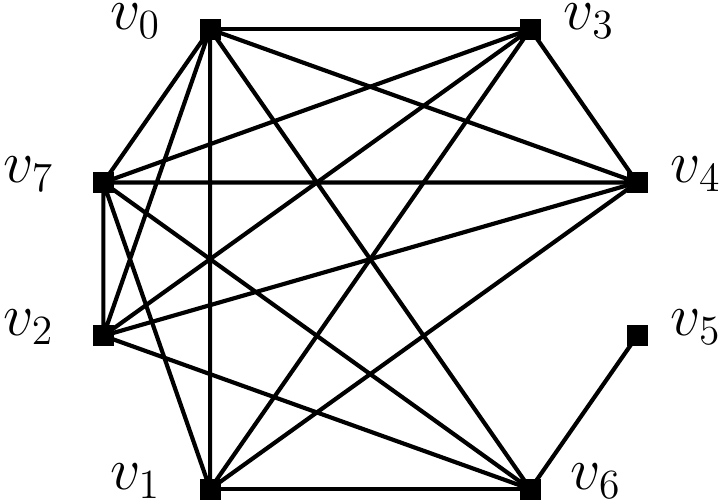}
 \caption{\label{fig:acessibility} Un arbre-de-graphes et le graphe d'accessibilité de ses feuilles.}
\end{figure}

Réciproquement $G$ se retrouve à partir d'une factorisation en arbre-de-graphes $T$ comme le \emph{graphe d'accessibilité} de ses feuilles, que nous définissons maintenant. L'ensemble de ses sommets $V_G$ correspond aux feuilles de $T$. Deux feuilles distinctes $v_1$, $v_2$ de $T$ définissent un plus court chemin dans $T$. Supposons qu'en chaque n\oe ud interne $k\in V_T$ emprunté par ce chemin, les sommets de $G_k$ qui s'envoient par $\varphi_k$ sur les arêtes du chemin incidentes à $k$ sont connectés par une arête. Alors, et seulement dans ce cas, les sommets $v_1,v_2\in V_G$ sont connectés par une arête. La figure \ref{fig:acessibility} montre le graphe d'accessibilité d'un arbre de graphes.

Ce procédé de factorisation en arbre-de-graphes dépend des choix des décompositions effectuées à chaque étape, mais il s'avère que l'arbre-de-graphe résultant est essentiellement unique au sens suivant. Disons qu'un arbre-de-graphes est \emph{réduit} si ses n\oe uds internes sont de degré au moins trois et s'il vérifie de plus les deux conditions suivantes. D'une part ses n\oe uds internes sont décorés par des graphes indécomposables ou dégénérés. Il faut d'autre part que deux cliques ne soient jamais reliées et qu'une arête de l'arbre entre deux étoiles relie leurs centres ou deux non-centres.
\begin{figure}[H]
 \centering
 \includegraphics[width=1.0\textwidth]{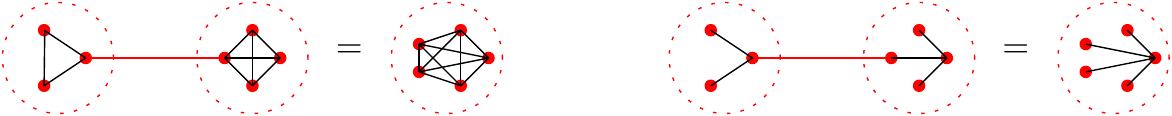}
 \caption{\label{fig:clique_star_join} Fusion des étoiles et des cliques}
\end{figure}
Cette dernière condition est nécessaire en vue d'assurer l'unicité de la factorisation dans le théorème qui suit. En effet, deux cliques $K_n$ et $K_m$ décorant des n\oe uds reliés dans l'arbre peuvent être fusionnées en une clique $K_{m+n-1}$, le graphe d'accesibilité sera inchangé. Il en est de même de la fusion de deux étoiles $S_n$ et $S_m$, reliées par un centre et un non-centre, en une étoile $S_{n+m-1}$.
Il peut sembler plus naturel de poursuivre la décomposition des graphes dégénérés en arbres de $K_3$ et $S_3$, cependant l'unicité n'est plus assurée puisque, par exemple, n'importe quel arbre de $K_3$ ayant $n$ feuilles fusionne en un $K_n$. 

\begin{thm}[Cunningham \cite{Cunning:1982}, Gioan-Paul \cite{GioPau:2012}]\label{thm_cunningham}
Tout graphe connexe est le graphe d'accessibilité d'un unique arbre-de-graphes réduit.
\end{thm}

\begin{rem}[structure d'opérade sur les graphes]
Cette décomposition en arbre-de-graphes révèle une structure d'opérade sur les graphes (étiquetés et enracinés en un sommet).
%
%
Nous verrons plus loin que les buissons se comportent bien pour l'opération de composition : une composition de buissons est un buisson, autrement dit ils forment une sous-opérade. C'est cette structure que nous allons exploiter pour en effectuer le dénombrement.

Un chapitre de \cite{Ghys:2017} introduit les opérades à travers plusieurs exemples en lien avec ce contexte ; la référence \cite{Fresse:2017} leur consacre une étude approfondie.
\end{rem}



\section{Combinatoire: buissons et diagrammes analytiques}

\subsection{Dénombrement des buissons étiquetés connexes}

Un graphe est \emph{complètement décomposable} si tout sous-graphe connexe induit ayant au moins quatre sommets est décomposable. Un graphe ayant moins de trois sommets vérifie trivialement cette propriété.

Cela implique qu'il admet une factorisation en arbre-de-graphes dont tous les n\oe uds sont décorés par des $K_3$ ou des $S_3$.
En fusionnant, l'arbre-de-graphes réduit a donc ses noeuds décorés par des graphes dégénérés.
Réciproquement, le graphe d'accessibilité d'un tel arrbre-de-graphes-dégénérés est totalement décomposable.

Par ailleurs, on montre facilement par récurrence dans \cite{GhySim:2020} que les graphes connexes complètement décomposables sont précisément les buissons connexes. On en déduit la description suivante des buissons qui va nous permettre de les énumérer.

\begin{Cor}\label{buisson_decomposition}
Un buisson connexe ayant au moins trois sommets est le graphe d'accessibilité d'un unique arbre-de-graphes-dégénérés réduit, et réciproquement. Nous les appellerons $SK$-arbres (clique-star graphs dans \cite{GioPau:2012}), et identifierons ces deux structures; la figure \ref{fig:acessibility} illustre cette correspondance.
\end{Cor}

\begin{rem}[Permutations séparables]
Cette caractérisation des buissons est analogue à celle des permutations séparables en termes de $(\oplus,\ominus)$-arbres.
Nous avons donc généralisé le fait, montré par \'Etienne Ghys dans \cite{Ghys:2017}, que les permutations séparables décrivent précisément les combinatoires possibles pour un germe de graphes polynomiaux du plan réel.
\end{rem}

Considérons la classe combinatoire $\mathcal{B}$ des buissons connexes dont les sommets sont étiquetés par les entiers d'un segment initial de $\N$, le sommet $0$ est désigné comme sa racine.
Appelons \emph{taille} du buisson le nombre de sommets différents de la racine (égale à l'étiquette maximale), et supposons qu'elle est non nulle; notre classe contient donc zéro tels buissons de taille $0$, un seul de taille $1$ qualifié d'élémentaire, et quatre de taille $2$.
Dès que la taille est au moins $2$, sa factorisation donne un $SK$-arbre dont une feuille est enracinée (les autres étiquetées), le \emph{n\oe ud racine} désigne celui qui est incident à la feuille racine.

\begin{figure}[H]
 \centering
 \includegraphics[width=0.7\textwidth]{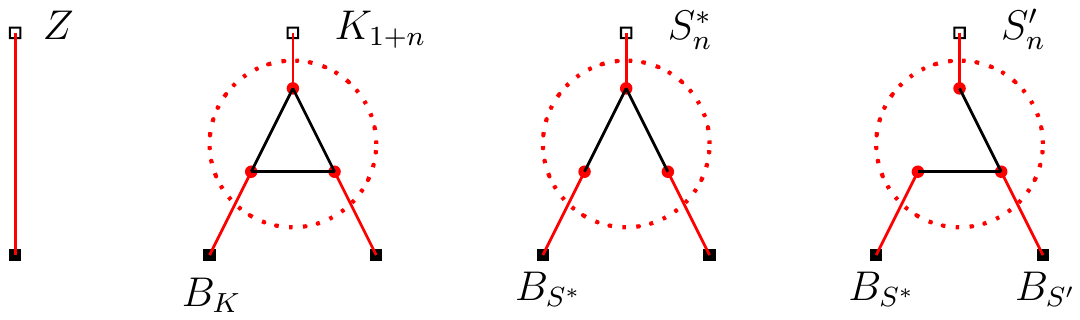}
 \caption{\label{fig:buisson_operad_root} Buisson élémentaire et types de n\oe uds racines avec les embranchements permis.}
\end{figure}

Soit $B(z)$ la série génératrice exponentielle de la classe $\mathcal{B}$ et $B_K$, $B_{S^*}$ et $B_{S'}$ celles qui énumèrent respectivement le nombre de $SK$-arbres n'ayant pas de clique au n\oe ud racine, n'ayant pas d'étoile dont le centre soit relié à la feuille racine, et n'ayant pas d'étoile dont un non-centre soit relié à la feuille racine.

L'unicité dans le théorème de Cunningham selon Gioan et Paul permet de dénombrer les buissons connexes décorés, en traduisant la construction itérative d'un $SK$-arbre enraciné en termes d'équations sur ces séries génératrices.

En effet remarquons tout d'abord qu'un tel arbre qui n'est pas élémentaire, possède comme n\oe ud racine une clique $K$, une étoile $S^*$ ou une étoile $S'$, donc $B(z)=\frac{1}{2}(B_K+B_{S^*}+B_{S'}-z)$.
De plus, un $SK$-arbre dont la racine n'est pas (par exemple) une clique, est élémentaire ou bien possède une racine qui est une étoile $S^*$ ou $S'$ avec $n>1$ tiges libres. Les tiges de $S^*$ portent des arbres n'ayant pas de $S^*$ à la racine, tandis que parmi celles de $S'$, il y en a une qui porte un arbre n'ayant pas de $S'$ à la racine et les autres portent des arbres n'ayant pas de $S^*$ à la racine. En prenant garde à diviser par la factorielle du nombre de tiges jouant le même rôle, on obtient la première équation du système suivant, les autres s'en déduisent de même :

\begin{align} \label{systemB1}
 &B_K(z)=z+\sum_{n>1}{\frac{B_{S^*}^n}{n!}}+\sum_{n>1}{B_{S'}\frac{B_{S^*}^{n-1}}{(n-1)!}}=z+(1+B_{S'})\exp{B_{S^*}}-B_{S^*}-B_{S'}-1
 \\
 \label{systemB2}
 &B_{S^*}(z)=z+\sum_{n>1}{\frac{B_K^n}{n!}}+\sum_{n>1}{B_{S'}\frac{B_{S^*}^{n-1}}{(n-1)!}}=z+\exp{B_K}+B_{S'}\exp{B_{S^*}}-B_K-B_{S^*}-1
 \\
 \label{systemB3}
 &B_{S'}(z)=z+\sum_{n>1}{\frac{B_K^n}{n!}}+\sum_{n>1}{\frac{B_{S^*}^{n}}{n!}}=z+\exp{B_K}+\exp{B_{S^*}}-B_K-B_{S^*}-2
\end{align}

\begin{Thm}
Le nombre de buissons connexes étiquetés de taille $n$ équivaut à
\begin{equation}\label{asymptoB}\tag{$\sim_B$}
\frac{b_0}{2\sqrt{\pi n^3}}.\beta^{-n}.n!
\end{equation}
où $\beta=2\sqrt{3}-1+2\log{\frac{1+\sqrt{3}}{2}}$ dont l'inverse vérifie $6.26<\beta^{-1}<6.27$,
et $b_0=\sqrt{\frac{\beta}{\sqrt{3}}}$.
Voici les premiers termes:
\begin{gather*}
    0,\, 1,\, 4,\, 38,\, 596,\, 13072,\, 368488,\, 12693536,\, 516718112,\, 24268858144,\\ 1291777104256,\, 76845808729472,\, 5052555752407424
\end{gather*}
\end{Thm}

\begin{proof}
Les expressions \ref{systemB1} et \ref{systemB2} donnent $B_K=B_{S^*}$, puis \ref{systemB2} et \ref{systemB3} : $B_{S'}=1-\exp(-B_K)$, et substituant dans \ref{systemB2} on a $B_K=z+f(B_K)$ où $f(w)=2\exp(w)+\exp(-w)-w-3$. On reconnait là un \og smooth implicit-function schema\fg{} donc d'après \cite[Thm VII.3]{FlajoSedge:2009}:
les coefficients de $B_K$ vérifient (\ref{asymptoB}) où $\beta$, son rayon de convergence, satisfait $f'(B_k(\beta))=1$. 
En posant $s=B_K(\beta)$ on trouve $2=2e^s-e^{-s}$, une équation du second degré en $e^s$, dont on cherche la solution positive: $s=\log{\frac{1+\sqrt{3}}{2}}$. En évaluant $B_K=z+f(B_K)$ en  $z=\beta$, on trouve $\beta=2s+1+2e^{-s}$. Ensuite, \cite[Thm VII.3]{FlajoSedge:2009} donne la valeur de $b_0$ en fonction de $\beta$ et $s$.
Enfin, $B=(g(B_K)-z)/2$ avec $g(w)=2w+1-\exp(-w)$ donc ses coefficients vérifient (\ref{asymptoB}); et leur expression exacte se déduit de l'inversion de Lagrange:
\[
g(B_K(z))=g(z)+\sum_{k=1}^\infty{\frac{1}{k!}\left(\frac{\partial}{\partial z}\right)^{k-1}\left( g'(z)f(z)^k\right)}
.\]
Comme $f(z)=\frac{3}{2}z^2+o(z^2)$, la somme jusqu'au rang $n$ fournit un calcul effectif des $n$ premiers termes du développement.
\end{proof}

\begin{rem}[cohérence avec l'OEIS]
La suite des premiers termes est par ailleurs connue de l'OEIS, recensée à la référence \hyperlink{http://oeis.org/A277869}{A277869} \cite{OEIS}.
\end{rem}

\begin{rem}[Buissons non étiquetés et automorphismes]
Pour en déduire une asymptotique du nombre de buissons non étiquetés il faudrait connaître la taille typique de leurs groupes d'automorphismes. Pour cela il serait bon d'étudier la forme typique de l'arbre-de-graphes et la distribution des degrés des sommets pour décomposer les symétries des buissons.
\end{rem}

\subsection{Structure et dénombrement des diagrammes analytiques}

\subsubsection{Opérade des diagrammes de cordes analytiques connexes enracinés}

Les graphes dégénérés décrivent l'entrelacement d'un unique diagramme de cordes (figure \ref{fig:degenrate_cordiag_root} en ignorant le marquage). Bouchet montra dans \cite{Bouchet:1987} que c'est également le cas des graphes indécomposables. La décomposition de Cunningham permet donc de montrer, comme dans \cite[section 4.8.5]{ChDuMo:2012}, que deux diagrammes de cordes ayant le même graphe d'entrelacement diffèrent par une série de mutations. Une \emph{mutation} consiste à appliquer une symétrie du groupe rectangulaire $(\Z/2\Z)^2$ à un sous-diagramme de cordes définit par deux intervalles sur le cercle (qui peuvent être vides ou consécutifs). En effet, la seule ambiguïté apparaissant lors de la reconstruction d'un diagramme de cordes à partir de la factorisation du buisson décrivant son entrelacement, provient du choix de l'orientation lors du remplacement d'une corde par un sous-diagramme de cordes. 

\begin{figure}[H]
 \centering
 \includegraphics[width=0.7\textwidth]{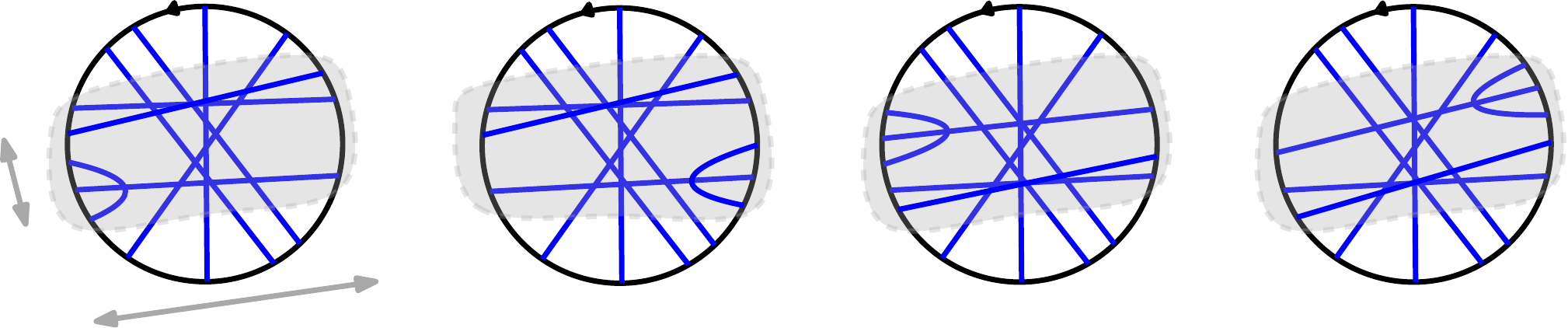}
 \caption{\label{fig:cordiag_mutation} Mutations d'un sous-diagramme défini par deux intervalles.}
\end{figure}

Considérons la classe combinatoire $\mathcal{C}$ des diagrammes de cordes analytiques \emph{connexes} et enracinés. La connexité d'un diagramme est définie comme celle de son graphe d'entrelacement. L'enracinement désigne ici le choix d'une corde $v$ distinguée que l'on oriente : cela revient à marquer un seul caractère $v^+$ du mot cyclique associé. Les cordes sont donc ordonnées et orientées par le parcours dans le sens positif depuis la tête $v^+$ de la racine.
Appelons \emph{taille} du diagramme le nombre de cordes non enracinées et supposons qu'elle est non nulle. Notre objectif est de décrire la combinatoire de cette classe et d'estimer le nombre de tels diagrammes quand la taille tend vers l'infini.
Pour cela, définissons une \emph{poulie} comme l'un des trois diagrammes enracinés de la figure \ref{fig:degenrate_cordiag_root}: $T_n$, $D^*_n$ et $D'_{k,l}$ pour $n>1$ et $k+l>0$. Une poulie enracinée en $v$ possède un côté gauche de $v^+$ à $v^-$ et un côté droit de $v^-$ à $v^+$.

\begin{figure}[H]
 \centering
 \includegraphics[width=1.0\textwidth]{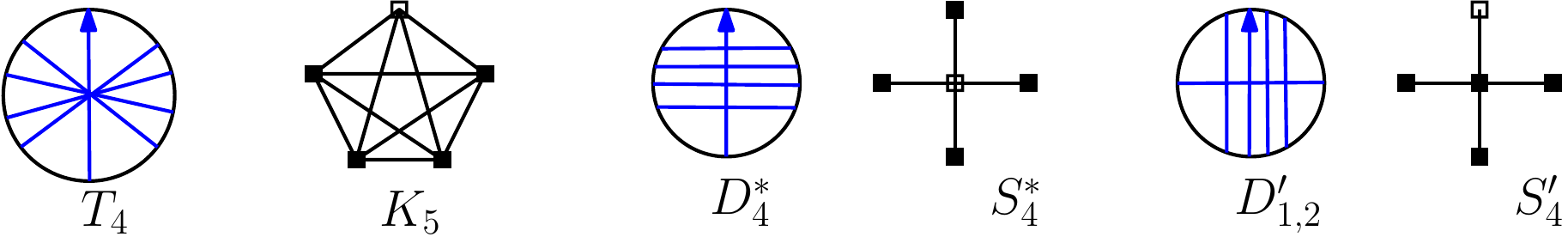}
 \caption{\label{fig:degenrate_cordiag_root} Les trois types de poulies et leurs graphes dégénérés associés.}
\end{figure}

Rappelons que l'enracinement d'un arbre définit un ordre partiel sur ses sommets $V_T$, dit \emph{généalogique}, dont les segments initiaux sont les géodésiques partant de la racine.
Si de plus l'arbre est plan, il en découle un ordre linéaire privilégié sur chaque ensemble de voisins d'un n\oe ud donné.

Appelons \emph{cordage} la donnée d'un arbre plan $T$ enraciné en une feuille, dont chaque n\oe ud interne $x$ est décoré par une poulie $P_x$. La structure plane induit une bijection $\varphi_x$ entre l'ensemble des cordes de $P_x$ et l'ensemble des voisins de $x$, envoyant la corde enracinée sur le prédécesseur immédiat de $x$ pour l'ordre généalogique susmentionné.

Un cordage sera dit \emph{réduit} s'il vérifie les conditions suivantes, portant sur les types de poulies associés à deux n\oe uds de l'arbre liés par une arête : deux poulies de type $T$ ne sont pas connectées, deux poulies de type $D^*$ ne sont pas connectées, et si $P_x$ est de type $D'$, nous demandons que $\varphi_x$ n'envoie pas la corde intersectant la racine sur un diagramme de type $D'$, et n'envoie pas les cordes parallèles à la racine sur des diagrammes du type $D^*$.  Les feuilles sont décorées par l'unique diagramme enraciné connexe de taille $1$.

Un cordage se \emph{contracte} en insérant dans chaque corde non enracinée $c$ de $P_x$ la poulie $P_{\varphi_x(c)}$, appliquant son intervalle de gauche sur $c^-$ et son intervalle de droite sur $c^+$, comme indiqué sur la figure \ref{fig:cord_operad}. Cette contraction fournit un diagramme de cordes analytique enraciné.

\begin{Lem}\label{unique_cordage}
Un diagramme analytique enraciné connexe provient d'un unique cordage.
\end{Lem}

\begin{proof}[Preuve]
\emph{Existence:} La factorisation $T(G)$ de son graphe d'entrelacement $G$ est un $SK$-arbre enraciné en une feuille. Enracinons les décorations $G_x$ de ses n\oe uds internes au sommet qui est le plus proche de la feuille racine. Nous obtenons alors un cordage $T(D)$ associé à $D$ en remplaçant les étiquettes des n\oe uds internes comme dans la figure \ref{fig:degenrate_cordiag_root}.

\emph{Unicité:}
Un cordage $T(D)$ associé $D$ fournit immédiatement la factorisation $T(G)$. L'unicité découle donc de celle du \ref{buisson_decomposition}, de l'unicité des diagrammes associés aux graphes dégénérés, et aux orientations imposées lors des insertions des poulies dans les cordes.
\end{proof}

\begin{figure}[H]
 \centering
 \includegraphics[width=0.6\textwidth]{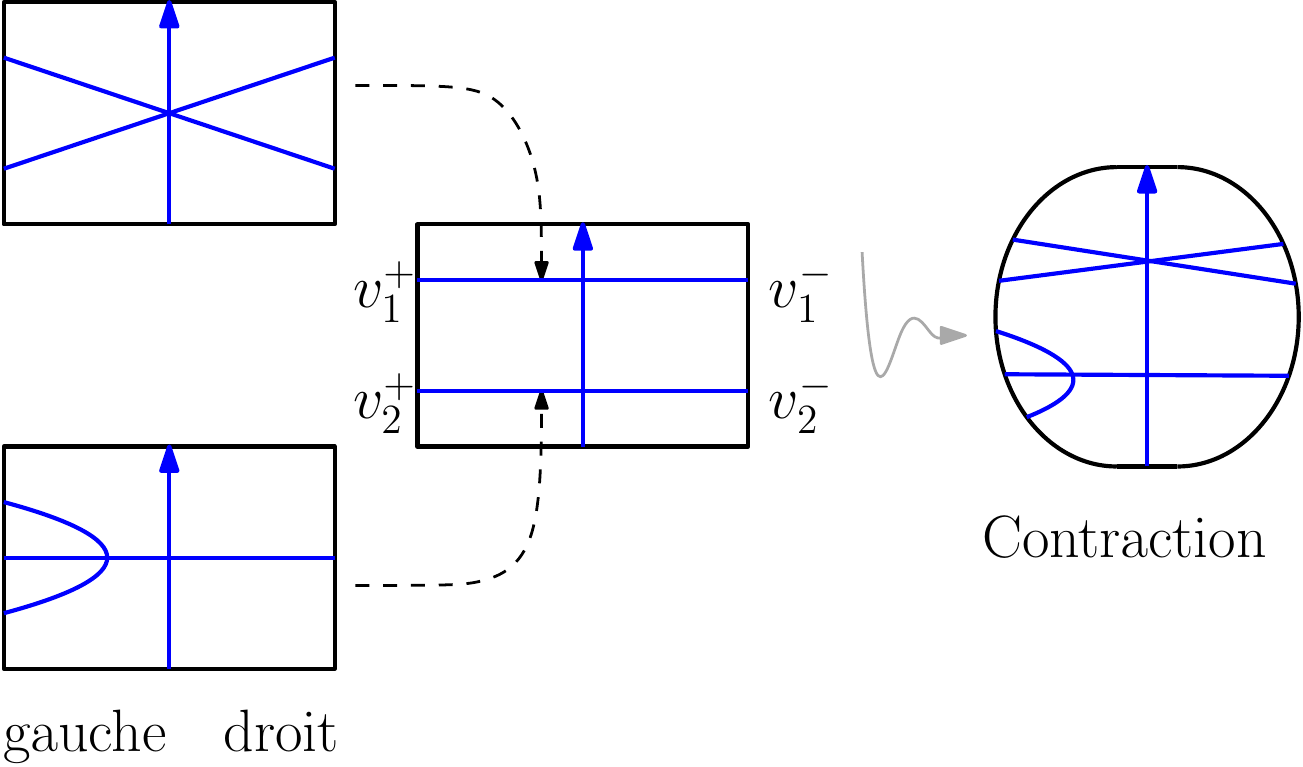}
 \hspace{0.5cm}
 \includegraphics[width=0.25\textwidth]{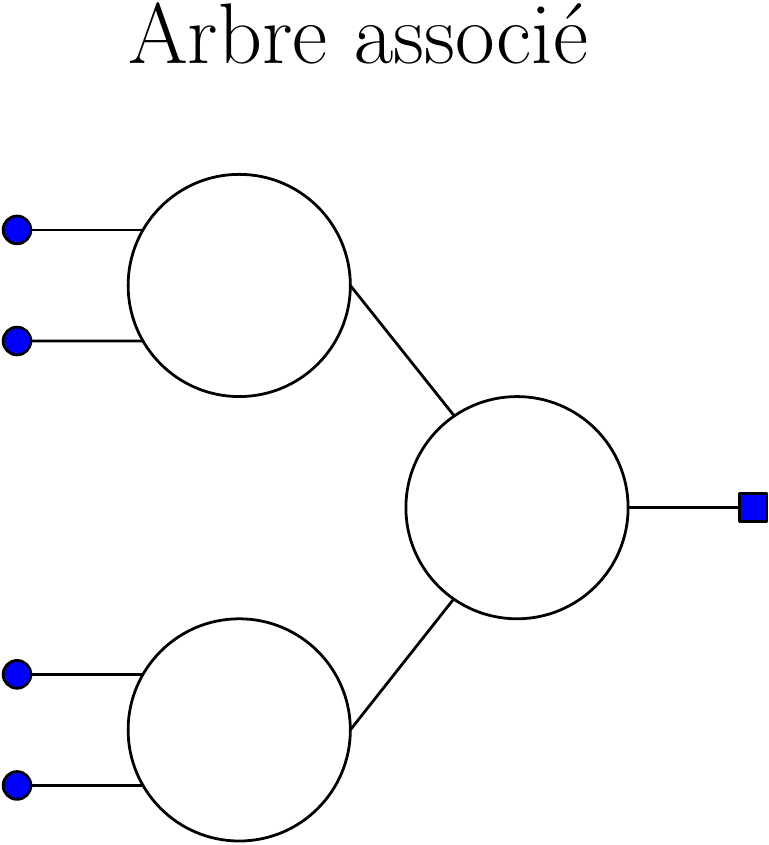}
 \caption{\label{fig:cord_operad} Composition dans l'opérade $\mathcal{C}$: insertion des poulies dans les cordes.}
\end{figure}

\subsubsection{Nombre de diagrammes de cordes analytiques connexes enracinés}
Ce lemme permet d'énumérer les objets de la classe $\mathcal{C}$ par taille. Soit $C(z)$ la série génératrice ordinaire de la classe $\mathcal{C}$; et $C_T$, $C_{D^*}$ et $C_{D'}$ celles qui énumèrent le nombre de cordages n'ayant pas pour n\oe ud racine une poulie du type indiqué par leur indice. Le lemme précédent assure l'unicité de la construction itérative donc $C(z) = \frac{1}{2}(C_T+C_{D^*} + C_{D'}-z)$ et:

\begin{align}\label{systemC1}
 &C_T(z) = z + \sum_{n>1}{C_{D^*}^n} + \sum_{k+l>0}{C_{D'}C_{D^*}^{k+l}} = z + \frac{C_{D^*}^2}{1-C_{D^*}} + C_{D'} \left[ \frac{1}{(1-C_{D^*})^2}-1\right]
 \\
 \label{systemC2}
 &C_{D^*}(z) = z + \sum_{n>1}{C_{T}^n} + \sum_{k+l>0}{C_{D'}C_{D^*}^{k+l}} = z + \frac{C_{T}^2}{1-C_{T}} + C_{D'} \left[ \frac{1}{(1-C_{D^*})^2}-1 \right]
 \\
 \label{systemC3}
 &C_{D'}(z) = z + \sum_{n>1}{C_{D^*}^n} + \sum_{n>1}{C_{T}^n} = z + \frac{C_{D^*}^2}{1-C_{D^*}} + \frac{C_{T}^2}{1-C_{T}}
\end{align}
De \ref{systemC1} et \ref{systemC2} on tire $ C_T = C_{D^*} $. On substitue dans \ref{systemC3} puis on calcule $ C = \frac{C_T}{1-C_T} $; et \ref{systemC1} donne après simplification: 
$ C_T^3 - 4 C_T^2 + (1+z) C_T - z=  0$. 
On en déduit la proposition suivante.

\begin{Prop} 
La série génératrice ordinaire $C$ du nombre $C_n$ de diagrammes de cordes analytiques connexes enracinés de taille $n$, est algébrique: $C=z+2zC+(z+2)C^2+2C^3$. Par conséquent $C_n \sim c_0\,n^{-\frac{3}{2}} \,\gamma^{-n}$ pour $c_0>0$ et $\gamma$ la racine réelle de $4x^3-49x^2+164x-12$:
\[\gamma =\frac{1}{12}\left(49-\frac{433}{\sqrt[3]{24407-1272\sqrt{318}}}-\sqrt[3]{24407-1272\sqrt{318}} \right).\]
Le nombre $C_n$ de diagrammes analytiques connexes enracinés, s'exprime par la formule:
\[
C_n=\frac{1}{n}\sum_{k=0}^{n-1}{\binom{n-1+k}{n-1}\binom{2n+k}{n-1-k}2^k}
.\]
Les premiers termes sont: $1, 4, 27, 226, 2116, 21218, 222851, 2420134, 26954622, 306203536$.
\end{Prop}

\begin{proof}
L'équation de $C$ s'obtient en substituant $C_T=\frac{C}{1+C}$ dans celle de $C_T$.
Les séries $C$ et $C_T$ ont le même rayon de convergence, car $C_T$ n'atteint pas la valeur $1$ avant sa première singularité.
Cette singularité $\gamma$ dicte la croissance exponentielle des coefficients des séries $C_T$ et $C$.
C'est la plus petite racine du discriminant selon la variable $y$ du polynôme $y^3-4y^2+(1+x)y-x$, définissant l'équation cubique pour $C_T$. Elle correspond en effet au premier point où cette cubique admet une tangente verticale, et n'est plus redevable du théorème des fonctions implicites. Le discriminant en question est l'équation qui apparaît dans l'ennoncé.
Remarquons en passant que $\gamma$ est bien l'unique racine de plus petit module du discriminant du polynôme: $2y^3+(x+2)y^2+(2x-1)y+x \in \C(x)[y]$ définissant l'équation sur $C$, confirmant le fait que $C$ et $C_T$ ont le même rayon.
L'asymptotique annoncée découle à nouveau du \og smooth implicit-function schéma\fg{}  \cite[Thm VII.3]{FlajoSedge:2009}.

Sur un voisinage de l'origine l'équation pour $C$ peut s'écrire $z=\frac{C}{\varphi(C)}$ avec $\varphi(v)=\frac{(1+v)^2}{1-2v-2v^2}$; donc par inversion de Lagrange $C_{n}=\frac{1}{n}[v^{n-1}]\varphi(v)^n$. Pour calculer $\varphi^n$, après avoir développé son dénominateur en série entière, on distribue son numérateur aux termes de la somme et on développe le binôme de Newton:
\begin{align*}
\varphi(v)^n
&=(1+v)^{2n}\cdot \left( \frac{1}{1-2v(1+v)} \right)^n
= (1+v)^{2n}\cdot \sum_{k\in\N}{\binom{n-1+k}{n-1}(2v)^k(1+v)^k} \\
&=
\sum_{k\in\N}{\binom{n-1+k}{n-1}(2v)^k(1+v)^{2n+k}}
=
\sum_{j,k\in\N}{\binom{n-1+k}{n-1}2^k\binom{2n+k}{j}v^{k+j}}
\:.\end{align*}
On extrait alors le terme de $v^{n-1}$ et on élimine l'indice $j$ de la somme:
\[
C_n=\frac{1}{n}\sum_{k+j=n-1}{\binom{n-1+k}{n-1}\binom{2n+k}{j}2^k}
=\frac{1}{n}\sum_{k=0}^{n-1}{\binom{n-1+k}{n-1}\binom{2n+k}{n-1-k}2^k}
.\]
\end{proof}

\begin{rem}[Reformulation en terme de grammaire algébrique]
Reformulons le cheminement de ce paragraphe en termes de grammaires algébriques (aussi dénommées acontextuelles).
L'ensemble des diagrammes analytiques connexes enracinés forme un langage. La description de sa structure (d'opérade) en termes de poulies et de cordages en fournit une grammaire génératrice qui est algébrique: 

\begin{itemize}
    \item $S \;\longrightarrow \quad z \;\mid\; T_{1} \;\mid\; D^*_{1} \;\mid\; D'_{1}$
    
    \item $T_{1} \longrightarrow \quad z \;\mid\; t_l D^*D^*t_r \;\mid\; t_l D^*D't_r \;\mid\; t_l D'D't_r$
    \item $T \;\longrightarrow \quad z \;\mid\; t_l D^*D^*t_r \;\mid\; t_l D^*D't_r \;\mid\; t_l D'D't_r \;\mid\; TT$
    
    \item $D^*_{1} \longrightarrow \quad z \;\mid\; d^*_l TTd^*_r \;\mid\; d^*_l TD'd^*_r \;\mid\; d^*_l D'D'd^*_r$
    \item $D^* \longrightarrow \quad z \;\mid\; d^*_l TTd^*_r \;\mid\; d^*_l TD'd^*_r \;\mid\; d^*_l D'D'd^*_r \;\mid\; D^*D^*$
    
    \item $D'_{1} \longrightarrow \quad z \;\mid\; d'_l Ld'_mMd'_nRd'_r$ \hspace*{3.45cm} $ \forall M \in \{T_1,D^*\},\; \forall L,R \in \{T,D'\}$
    \item $D' \,\longrightarrow \quad z \;\mid\; d'_l Ld'_mMd'_nRd'_r \;\mid\; D'D'$ \hspace*{2cm}  $\forall M \in \{T_1,D^*\},\; \forall L,R \in \{T,D'\}$ 
\end{itemize}
Les indices $1$ permettent de distinguer les emplacements où la lettre en question ne peut être dupliquée: au n\oe ud racine ainsi que dans la corde de $D'$ intersectant la corde racine.

Le lemme \ref{unique_cordage} signifie que cette grammaire est non ambigüe: chaque diagramme n'est engendré qu'une fois.
On peut donc en tirer un système d'équations algébriques (plus volumineux que le système utilisé (\ref{systemC1},\ref{systemC2},\ref{systemC3}) qui n'est pas algébrique), dont le diagramme de dépendance montre qu'il est irréductible (voir \cite[VII.6.3]{FlajoSedge:2009} pour les notions de ce paragraphe). Les solutions sont des séries apériodiques puisque leurs coefficients sont positifs à partir d'un certain rang, et le théorème \cite[VII.5]{FlajoSedge:2009} permet d'en déduire l'asymptotique de leurs coefficients.
De manière analogue, le système d'équations sur les séries $C_T$, $C_{D^*}$, $C_{D'}$ vérifie les conditions du théorème de Drmota-Lalley-Woods \cite[VII.6]{FlajoSedge:2009}.
\end{rem}

\begin{rem}[Développement asymptotique des coefficients]
Ce théorème de Drmota-Lalley-Woods prévoit un développement asymptotique complet des $C_n$ de la forme suivante:
\[
\gamma^{-n}n^{-\frac{3}{2}}\left( \sum_{k\in \N}{c_k\,n^{-k}}\right)
.\]
Le développement de l'équation au voisinage du point critique permettrait de calculer les $c_k$.
\end{rem}

\begin{quest}[Interprétation combinatoire de l'équation sur $C$]
L'équation sur la série génératrice ordinaire des diagrammes de cordes analytiques connexes enracinés peut s'écrire sous la forme $C=z+2zC+(z+2)C^2+2C^3$. Le terme de gauche est la série $C$ recherchée, tandis qu'à droite nous voyons un polynôme à coefficients entiers positifs en la variable $z$ et la série $C$. C'est le genre d'identités qui découle d'une grammaire algébrique \cite{BanDrmota:2015} : il est donc naturel d'espérer pouvoir en trouver une interprétation combinatoire, permettant alors de la déduire instantanément.
\end{quest}

\subsubsection{Dénombrement des diagrammes de cordes analytiques enracinés}
Enumérons désormais la classe combinatoire $\mathcal{A}$ des diagrammes de cordes analytiques enracinés, comme précédemment, en un \emph{rayon}, c'est-à-dire en une lettre du mot cyclique. Remarquons que cet enracinement revient à considérer des \emph{diagrammes linéaires}, autrement dit des suites de $2n$ lettres, chaque lettre apparaissant deux fois. La \emph{taille} est toujours définie comme le nombre de cordes et nous autorissons cette fois-ci que la taille soit nulle. Les diagrammes à zéro et une corde sont inclus.

\begin{rem}[diagrammes enracinés versus linéaires]
Pour $n>1$ nous avons a une bijection entre les diagrammes linéaires de taille $n$ et les diagrammes enracinés de taille $n-1$, obtenue en enracinant la premiere lettre située après le point marqué.
Diagrammes enracinés et linéaires sont donc les mêmes objets, à ceci près que leurs fonctions taille sont décalées de un, et que les conventions sur les petits objets diffèrent.

Selon le contexte et le type de construction récursive il peut être parfois commode de jongler entre les deux façons de penser: la première se prête bien à l'insertion de diagrammes à la place des cordes d'un autre (structure d'opérade), tandis que le second est adéquat pour la concaténation de diagrammes les uns à la suite des autres (structure séquentielle, ou de somme connexe).
\end{rem}

Soit $\mathcal{L}$ la classe combinatoire des diagrammes analytiques connexes linéaires, c'est-à-dire marqués en un point du cercle; et dont la taille est le nombre de cordes. Nous incluons cette fois-ci le diagramme à une corde mais pas celui à zéro cordes. Sa série génératrice est $L(z)=z+zC(z)$, toujours d'après la remarque précédente. Par conséquent $L$ est algébrique elle aussi, vérifiant $2L^3+(z^2-4z)L^2+z^2L+z^3=0$, et ses coefficients sont en $\Theta(n^{-\frac{3}{2}} \gamma^{-n})$.

\begin{Lem}
Les series generatrices $A$ et $L$ sont liées par l'équation $A=1+L(zA^2)$.
\end{Lem}

\begin{proof}

Disons qu'une corde $a$ d'un diagramme linéaire est \emph{recouverte} par la corde $b$ si $b^+$ vient avant $a^+$ et $b^-$ vient après $a^-$. C'est une relation d'ordre partielle sur l'ensemble des cordes.
Désignons par \emph{composante connexe} du diagramme, les cordes correspondant à une composante connexe dans le graphe d'entrelacement.
Appelons \emph{rez-de-chaussée} d'un digramme linéaire, la suite des diagrammes formés par les composantes connexes des cordes non recouvertes, c'est-à-dire maximale pour cette relation de recouvrement.
Un diagramme de cordes analytique linéaire est formé de son rez-de-chaussée, et de diagrammes linéaires qui viennent s'imbriquer dans les interstices délimités par les cordes d'une même composante connexe.
\begin{figure}[H]
 \centering
 \includegraphics[width=0.9\textwidth]{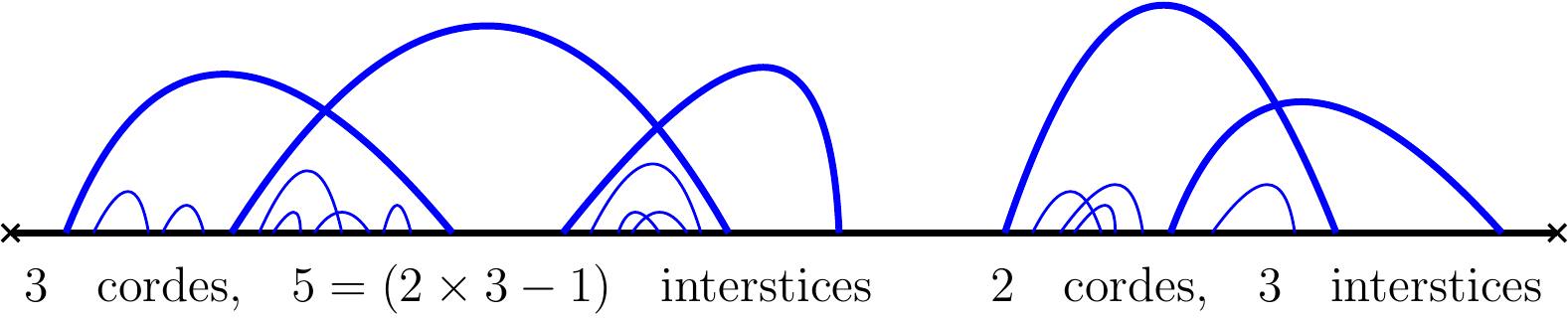}
 \caption{\label{fig:linear_diagram_first_floor_decomposition} Rez-de-chaussée du diagramme en gras, sous diagrammes dans les interstices}
\end{figure}
\noindent Un élément de la classe $\mathcal{A}$ peut contenir zéro cordes, sinon son rez-de-chaussée possède $k>0$ composantes connexes non vides, éléments de la classe $\mathcal{L}$. Chaque composante connexe recouvre un certain nombre (dépendant de sa taille comme expliqué ci-dessus) d'éléments de la classe $\mathcal{A}$. Nous utilisons ici de manière cruciale les conventions adoptées pour les petits objets de nos classes. On en déduit la relation:
\[A=1+\sum_{k\in \N^*}{ \sum_{(l_1,\dots l_k)\in\mathcal{L}^k}{ \prod_{j=1}^{k}{ z^{\lvert l_j\rvert}A^{2\lvert l_j\rvert-1} } } } .\]
La série $A$ est inversible pour la multiplication dans l'anneau des séries formelles (ou des séries convergentes) car son coefficient constant est $1$. En utilisant ce fait, on peut factoriser chaque terme de la somme sur $k$ pour reconnaître ensuite l'évaluation de la série $L$ en $zA(z)^2$:
\[
A-1=\sum_{k\in \N^*}{\frac{1}{A^k}\sum_{(l_1,\dots l_k)\in\mathcal{L}^k}{ \prod_{j=1}^{k}{ (zA^2)^{\lvert l_j\rvert} } } }
=\sum_{k\in \N^*}{\frac{1}{A^k}\left[ \sum_{l\in\mathcal{L}}{ (zA^2)^{\lvert l\rvert} } \right]^k }= \sum_{k\in \N^*}{\frac{1}{A^k}L(zA^2)^k }
.\]
Par conséquent,
\[
A=\sum_{k\in \N}{\left(\frac{L(zA^2)}{A}\right)^k}
= \frac{A}{A-L(zA^2)}
.\]
Et en utilisant à nouveau l'inversibilité de $A$ on obtient bien $A=1+L(zA^2)$.
\end{proof}

\begin{Thm}\label{diaganalin}
La série génératrice ordinaire $A$ comptant le nombre $A_n$ de diagrammes analytiques enracinés est algébrique:
\[(z^3+z^2)A^6-z^2A^5-4zA^4+(8z+2)A^3-(4z+6)A^2+6A-2=0 .\]
Ainsi $A_n\sim a_0\,n^{-\frac{3}{2}}\,\alpha^{-n}$ pour $4<10^3\,a_0<5$ et $\alpha$ la plus petite racine du discrimant de cette expression en tant qu'élément de $\R(z)[A]$. Elle vérifie $0.063321613 < \alpha < 0.063321614$, et son inverse $15.792395< \alpha^{-1}< 15.792396$.
Les premiers termes de la suite $A_n$ sont:
\begin{gather*}
    1,1,3,15,105,923,9417,105815,1267681,15875631,205301361
\end{gather*}
\end{Thm}

\begin{proof} On évalue l'équation vérifiée par $L$ en $zA^2$ pour obtenir celle sur $A$. Ensuite la preuve est similaire aux deux précédentes. Une assistance informatique permet d'encadrer la plus petite racine $\alpha$ du discriminant de cette équation. On identifie la branche de la courbe qui correspond à la série $A$ grâce aux premiers termes obtenus avec l'équation par la méthode des coefficients indéterminés.
Cette fois-ci on applique le \cite[Théorème VII.8]{FlajoSedge:2009} pour obtenir une asymptotique en $A_n\sim a_0\,n^{-1-k_0/\kappa}\alpha^{-n}$ où dans notre cas: $k_0=1$ est l'ordre de $y(z)$ en $z=0$, et $\kappa=2$ vaut la multiplicité plus un de $\alpha$ en tant que racine du discriminant. Enfin, le théorème \cite[Thm VII.3]{FlajoSedge:2009} permet de calculer $a_0$ en fonction de $\alpha$ et $A(\alpha)$ et on trouve informatiquement l'encadrement $4<10^3\,a_0<5$.
\end{proof}

\begin{rem}[Vérification algorithmique des neuf premiers coefficients]
Après toutes ces manipulations combinatoires il est naturel de vérifier si les premiers termes de la série génératrice coïncident effectivement avec le nombre de diagrammes linéaires analytiques.
La description récursive des diagrammes analytiques nous a permis de tous les engendrer et de confirmer algorithmiquement les neuf premiers coefficients annoncés. Ceci constitue une vérification expérimentale du travail aboutissant aux équations sur les séries $C$ et $A$.
\end{rem}

\subsubsection{Asymptotique du nombre de diagrammes de cordes analytiques} Etudions désormais la classe $\Tilde{\mathcal{A}}$ des diagrammes de cordes analytiques; c'est-à-dire des mots dans lesquels chaque lettre apparaît deux fois, considérés à permutation cyclique près, qui décrivent la topologie des singularités de courbes analytiques du plan orienté.
Le groupe cyclique à $2n$ éléments agit par rotation sur l'ensemble $\mathcal{A}_n$ des diagrammes linéaires analytiques à $n$ cordes, et les diagrammes de cordes analytiques correspondent aux orbites. Notons $\mathcal{A}^d_n$ l'ensemble des diagrammes linéaires à $n$ cordes dont le stabilisateur pour l'action cyclique est de cardinal $d$, et $A^d_n$ son cardinal. D'après la formule des classes:
\begin{equation}\label{formule_classes}
\Tilde{A}_n=\sum_{w\in \mathcal{A}_n}{\frac{1}{\lvert \omega(w)\rvert}}
=\sum_{w\in \mathcal{A}_n}{\frac{\lvert \mathrm{Stab}(w)\rvert}{n}}
=\frac{1}{n}\sum_{d\mid 2n}{d\, A^d_n}
\:.\end{equation}

L'objectif est désormais d'estimer les quantités $A^d_n$ selon la valeur de $d\mid 2n$ en vue de montrer que $A_n\sim A^1_n$. Pour celà, commençons par discuter la notion de diagramme quotient.

\paragraph{Inversions de cordes, diagramme quotient, monodromies.}
Si un diagramme linéaire $w$ possède une symétrie de rotation $d\in \mathrm{Stab}(w)$ qui envoie l'extrémité d'une corde, c'est-à-dire une lettre du mot cyclique, sur l'autre extrémité de la même corde; alors nous dirons que cette symétrie \emph{inverse une corde} du diagramme.
Remarquons qu'elle est nécessairement d'ordre deux: $d=n\in \Z/2n\Z$.
Supposons au contraire que $d$ n'inverse pas de cordes, alors 
\emph{le quotient de $w$ par l'action de $d$}, autrement dit son orbite par $d$, s'identifie à un diagramme linéaire $\overline{w}=w \pmod{d}$ de taille $t=n/d$.
On peut alors reconstruire $w$ à partir de son quotient $\overline{w}$ dont chaque corde est décorée d'un entier modulo $\frac{n}{t}$: sa \emph{monodromie}. Concrètement (mais abusivement), on considère les entiers numérotés de $0$ à $2n-1$, et on apparie les points $i<j$ si $i \pmod{t}$ et $j \pmod{t}$ sont appariés dans $\overline{w}$ et si $\lfloor \frac{j-i}{t} \rfloor$ est égal à la monomdromie de la corde correspondante dans $\overline{w}$.

\begin{Lem}[Le quotient préserve l'analycité]
Si $w\in \mathcal{A}_n$ possède une symétrie $d$ qui n'inverse pas de cordes, alors $w \pmod{d} \in \mathcal{A}_{t}$.
\end{Lem}

\begin{proof}
Fixons la rotation $d$ et raisonnons par récurrence sur le nombre de cordes du diagramme quotient. S'il en a moins de quatre il est analytique. Sinon, le diagramme $w$ possède une corde isolée, une fourche ou une paire de jumeaux comme sur la figure \ref{fig:isole_fourche_jumeaux}. Ce motif est préservé par la symétrie et il passe donc au quotient en un motif similaire dans $\overline{w}$. On peut donc le simplifier pour obtenir $\overline{w}'$. Mais ce diagramme est aussi un quotient du diagramme analytique $w'$ obtenu en simplifiant tous les motifs dans une même orbite par la symétrie. L'hypothèse de récurrence implique que $\overline{w}'$ est analytique, donc par la caractérisation récursive de l'analycité $\overline{w}$ aussi.
\end{proof}
\begin{Lem}
\[\sum_{2<d\leq 2n}{A^d_n}
\leq \sum_{2<d\leq 2n}{d \, A^d_n}
=o(8^n)
.\]
\end{Lem}

\begin{proof}
Pour $d>2$, le quotient de $w\in \mathcal{A}^d_n$ est bien défini et c'est un diagramme de corde analytique. Comme $w$ est alors uniquement déterminé par son quotient $w \pmod{d}$ décoré des monodromies de ses cordes, on a:
$A_n^d\leq \left(\frac{n}{t}\right)^{t} A_{t}$.
Or $\sup \{\left(\frac{n}{t}\right)^{t}\vert 0<t\}=e^{\frac{n}{e}}<2^n$.
Ainsi, en se souvenant que $\alpha^{-1}<16$:
\[
\sum_{2<d\leq 2n}{A^d_n}
\leq \sum_{2<d\leq 2n}{d \cdot A^d_n} 
= O\left( \sum_{2<d\leq 2n}{n 2^n n^{-\frac{3}{2}} \alpha^{-t}}\right)
= O\left(n^{\frac{1}{2}} 2^n \alpha^{-\frac{n}{2}} \right)
=o\left( 8^n \right)
.\]
\end{proof}

\begin{Lem}
$A^2_n=o(12^n)$.
\end{Lem}

\begin{proof}
On partitionne les éléments de $\mathcal{A}^2_n$ selon le nombre $k$ de cordes laissées fixes par la symétrie $n$ d'ordre $2$. En enlevant ces cordes, on peut passer au quotient pour obtenir un diagramme analytique $\overline{w}$ de taille plus petit que $\frac{n}{2}$. Le diagramme $w$ initial est uniquement déterminé par l'emplacement des $k$ cordes fixes par l'involution, et par le quotient $\overline{w}$ muni des monodromies (modulo $2$) de ses cordes. La symétrie du diagramme implique que les cordes fixes sont déterminées par les $k$ extremitées situées dans les $n$ premières lettres de $w$. Par conséquent, en se rappelant que $\alpha^{-1}<16$, on a bien:
\[
A^2_n 
\leq \sum_{k=0}^{n} \binom{n}{k}2^{\frac{n}{2}}A_{\frac{n}{2}}
= O\left( 2^n 2^{\frac{n}{2}} \alpha^{-\frac{n}{2}} \right)
=o(12^n)
.\]
\end{proof}

\begin{Prop}
Le nombre $\Tilde{A}_n$ de diagrammes de cordes analytiques vérifie:
\[
\Tilde{A}_n \sim \frac{A_n}{n} 
\sim a_0 \, n^{-\frac{5}{2}}\, \alpha^{-n}
.\]
\end{Prop}

\begin{proof}
D'après les deux lemmes précedents on a: $A_n=\sum_{d}{A^d_n}=A^1_n+o(12^n)$ or $12^n = o(A_n)$ donc $A^1_n\sim A_n$.
En appliquant tout ce qui précède à l'identité \ref{formule_classes}:
\[\Tilde{A}_n
=\frac{1}{n} \sum_{d\mid 2n}{d\cdot A^d_n}
=\frac{A^1_n}{n}+o(12^n)
\sim \frac{A^1_n}{n}
\sim \frac{A_n}{n}
.\]
\end{proof}

\begin{rem}[Diagrammes analytiques diédraux]
On prouve de manière analogue que le nombre de diagrammes anlytiques modulo symétrie diédrale équivaut à $\frac{A_n}{2}$.
\end{rem}

\section{Courbes analytiques réelles singulières}

Rappelons qu'une surface analytique réelle $S$ est une variété de dimension $2$ donnée par un atlas dont les changements de cartes sont analytiques, et qu'une courbe analytique est le lieu d'annulation d'une fonction analytique $f\colon S \to \R$.

Dans cette section nous décrivons d'abord la topologie globale d'une courbe analytique réelle singulière sur une surface analytique réelle : après avoir introduit le concept de courbe combinatoire, nous verrons qu'il n'y pas d'obstructions globales quant à sa réalisation par des courbes analytiques. Nous énumèrons ensuite les types topologiques possibles pour une courbe analytique réelle de la sphère en fonction du degré topologique de ses singularités. Enfin, nous majorons cette quantité par une fonction du nombre d'arêtes uniquement.

\subsection{Topologie globale : en analytique il n'y a pas d'obstructions}

\subsubsection{Question directrice.}

Une courbe analytique réelle $\gamma$ sur la sphère possède un nombre fini de singularités, à chacune lui étant associée un diagramme de cordes.
Le dessin topologique est donc celui d'un certain nombre de points singuliers dont les rayons sont reliés par des arcs lisses disjoints.
Par ailleurs, si on se place en un point non singulier de la courbe et qu'on choisit une orientation locale de la courbe en ce point, il existe une unique manière de poursuivre son chemin y compris à travers les singularités. En effet, une branche qui arrive en un point singulier est appariée avec une autre qui en ressort. Au bout d'un certain temps on sera retourné au point de départ. On peut alors parcourir d'autres brins de la courbe $\gamma$.

Réciproquement imaginons un ensemble de germes de courbes analytiques singulières de la sphère dont les rayons sont reliés deux à deux par des arcs lisses disjoints. A quelles conditions ce dessin provient-il d'une courbe analytique ? Insistons sur le fait qu'on demande (implicitement dans l'expression \og type topologique \fg) non seulement que la courbe analytique soit homéomorphe à ce dessin mais aussi que les traversées des singularités soient les mêmes, autrement dit que les diagrammes de cordes soient les mêmes. Formalisons $\dots$

\begin{figure}[H]
 \centering
  \includegraphics[width=0.35\textwidth]{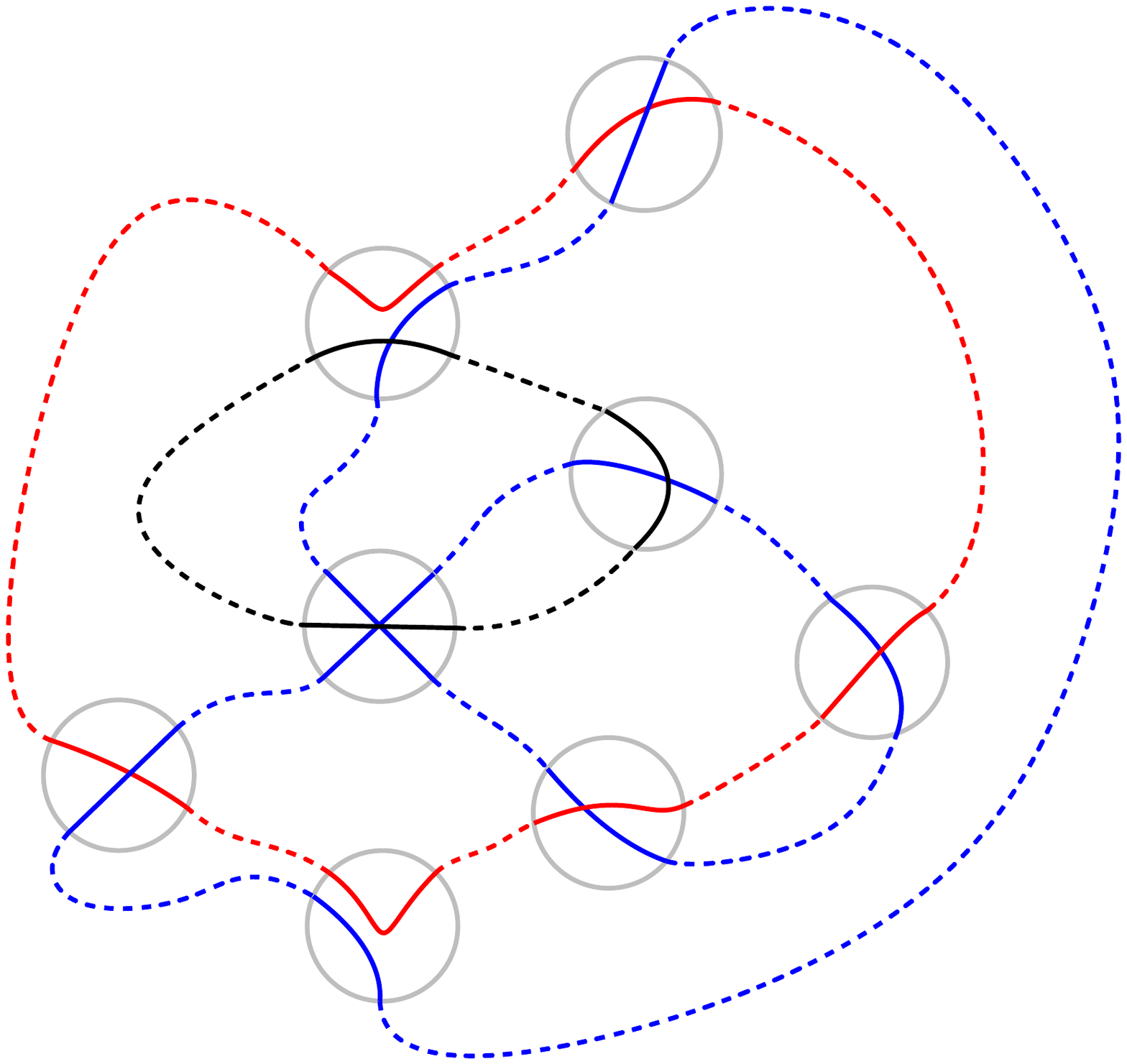}
  \includegraphics[width=0.6\textwidth]{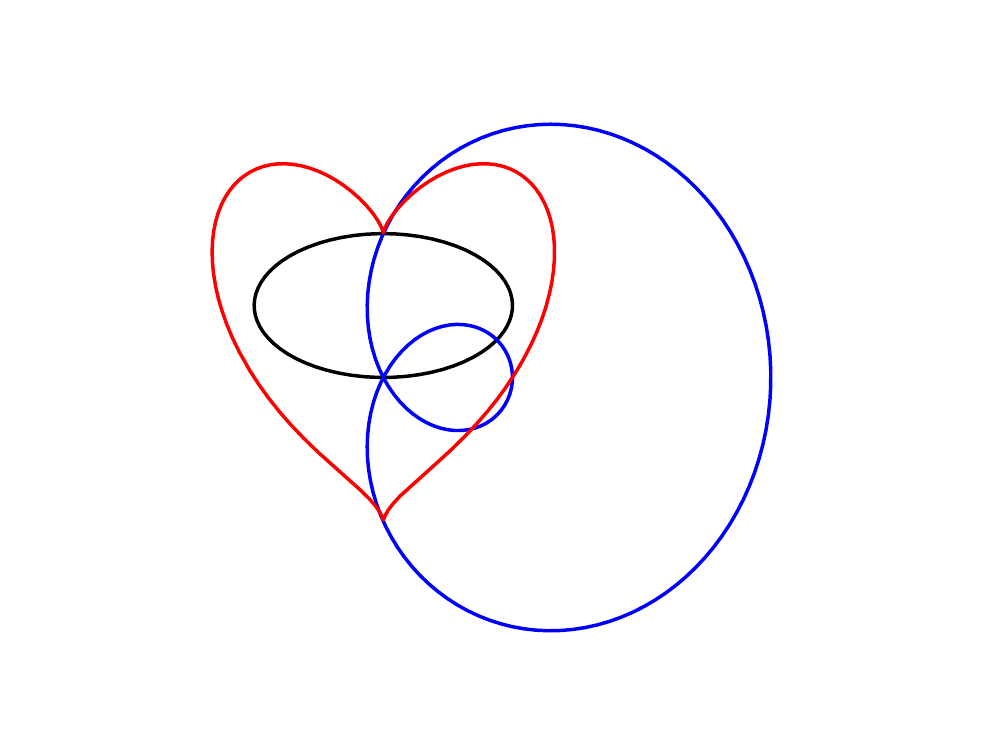}
 \caption{\label{fig:cardio_coeur_ellipse_py} Un dessin topologique provenant d'une courbe algébrique.}
\end{figure}

\subsubsection{Formulation combinatoire.}
\label{formulation_combinatoire}
Dans cette section nous adoptons une définition équivalente d'un diagramme de cordes $C$. C'est la donnée d'un ensemble $R$ de cardinal $2c$, qu'on appelle ses \emph{rayons}, muni d'une permutation cyclique $\sigma$ et d'une involution $\tau$ sans points fixes. Les orbites sous l'involution sont appelées ses \emph{cordes}.
%
On rappelle qu'un diagramme de corde est analytique si c'est le halo d'une singularité de courbe analytique réelle plane.

On utilise aussi une notion plus générale de \emph{graphe}. C'est une paire d'ensembles de sommets et de demi-arêtes $(V,E)$, munie d'applications extrémités $ e\in E \mapsto e_\pm \in V$; et d'une involution sans points fixes $\alpha \colon E \to E$ telle que $\alpha(e)_-=e_+$, dont les orbites sont les arêtes.

\begin{Define}[courbe combinatoire]
Une \emph{courbe combinatoire} est la donnée d'un ensemble de $s$ diagrammes de cordes $C_1,\dots,C_s$ et d'une involution sans points fixes $\alpha$ sur l'ensemble $R=R_1\cup \dots \cup R_s$ de leurs rayons.
\end{Define}

\paragraph{La carte associée.}
On peut associer à une telle courbe combinatoire un graphe $G$ dont les sommets sont les diagrammes de cordes et dont les demi-arêtes incidentes à un diagramme sont en bijection avec ses rayons. Quant aux arêtes, elles sont données par l'involution $\alpha$ qui apparie les demi-arêtes.
L'ensemble $R$ des demi-arêtes de $G$ est muni d'une permutation $\sigma=\sigma_1 \dots \sigma_s$ associée à la rotation simultanée des diagrammes ainsi que d'une involution $\alpha$ définie par l'appariement choisi entre les demi-arêtes: nous avons donc une carte combinatoire $(R,\sigma,\alpha)$.

\begin{rem}
Le livre \cite{LandoZvonkine:2004} étudie les apparitions de ces objets sous diverses formes et dans des contextes variés. Rappelons simplement que la structure de carte combinatoire, c'est-à-dire l'ensemble $R$ muni de deux permutation dont l'une est une involution sans points fixes, équivaut à celle faite à homéomorphisme près, d’un graphe plongé dans une surface, avec un complémentaire qui est une union disjointe de disques. Insistons : le genre de la surface se déduit de la donnée des permutations.
\end{rem}

\paragraph{Réciproquement...}
En plus de la carte, la courbe vient avec une involution supplémentaire sur les demi-arêtes: le produit $\tau=\tau_1 \dots \tau_s$ des involutions provenant de chaque diagramme de cordes. Réciproquement, la donnée $(R,\alpha,\sigma,\tau)$ d'une carte combinatoire et d'une involution sans points fixes qui préserve l'incidence des demi-arêtes (c'est-à-dire que les orbites selon $\tau$ forment une partition des orbites selon $\sigma$), équivaut à celle de la courbe combinatoire.

\begin{Define}[Brins]
Les deux involutions $\alpha$ et $\tau$ engendrent un sous-groupe du groupe symétrique $\mathfrak{S}(R)$ sur l'ensemble des demi-arêtes et nous définissons un \emph{brin} comme une orbite sous son action.
\end{Define}

\begin{ex}
Par exemple, toute application de $\kappa$ cercles dans une surface définissant un plongement lisse en dehors d'un nombre fini de singularités, induit une courbe combinatoire à $\kappa$ brins.
Une courbe analytique réelle définit également une courbe combinatoire, qualifiée \emph{d'analytique}.
Un brin correspond à l'idée naturelle que l'on se fait d'un parcours le long d'une courbe dont la traversée d'une singularité reste sur la même corde: lorsqu'on arrive à un sommet selon un certain rayon, on en ressort selon le rayon qui lui est apparié par le diagramme de cordes.
\end{ex}

\begin{rem}
Dans ces cas, il faut cependant prendre garde au fait que la surface retrouvée par la structure de courbe combinatoire ne sera pas nécessairement homéomorphe à la surface initiale ; ceci n'est assuré que lorsque le complémentaire de la courbe dans la surface initiale est une union de disques. Ce détail n'est pas problématique pour le théorème \ref{cocan} de réalisation à venir car la structure qui importe est celle du graphe de diagrammes de cordes plongé dans une surface. La notion de courbe combinatoire aura ici son intérêt pour les questions de dénombrement sur la sphère et le plan projectif.
\end{rem}

On peut désormais reformuler la question initiale : à quelles conditions une courbe combinatoire est-elle analytique ?
On dit que la courbe combinatoire vérifie \emph{l'hypothèse locale} si ses diagrammes de cordes sont analytiques.

\begin{Thm}[Courbes combinatoires analytiques] \label{cocan}
Toute courbe combinatoire dans une surface analytique réelles vérifiant l'hypothèse locale est associée à une courbe analytique.
\end{Thm}

\begin{proof}
\emph{0/ Deux résultats.}
Rappelons \cite[no. 8]{Cartan:1958} pour quelques éléments de géométrie analytique réelle.
Grauert a montré \cite{Grauert:1958} que toute surface analytique réelle $S$ admet un plongement propre analytique dans un espace numérique et le théorème d'approximation de Whitney \cite{Whitney:1934} affirme que sa dimension peut être prise égale à $4$. Pour alléger les notations, ces plongements analytiques seront fixés de manière à identifier les objets avec leurs images dans ces espaces numériques.
Par ailleurs, Cartan a montré \cite[no. 7.2]{Cartan:1957} qu'une courbe de $S$ définie localement comme le lieu d'annulation de fonctions analytiques est le lieu d'annulation d'une fonction analytique définie globalement sur $S$.

\emph{1/ Préparation du terrain.}
L'hypothèse locale permet de choisir pour chaque diagramme $C_j$ un germe de fonction analytique réelle $f_j$ définie sur un disque ouvert $U_j$ de $S$, ne contenant qu'un seul point critique $x_j$, et tel que $(U_j,\{f_j=0\})$ réalise $C_j$.
Prenons les disques $U_j\subset \R^2$ fermetures deux à deux disjoints; et choisissons des morceaux de courbes lisses de la surface, qui connectent sans s'intersecter les branches des germes $f_j=0$ selon l'appariement $\alpha$. Remarquons qu'on peut même choisir leur classe d'homotopie dans la surface $S$ relativement à l'union des $U_j$. Le tout forme une courbe $\xi$ qui est lisse et sans intersections en dehors des $x_j$, et qui coïncide avec les $f_j=0$ sur les $U_j$.

\emph{2/ \'Eclatement aux points singuliers pour se ramener au monde lisse.}
D'après le théorème de résolution des singularités de Noether, on peut éclater la surface au-dessus des $U_j$ de manière à résoudre la courbe $f_j=0$ comme dans \cite[1.3]{GhySim:2020}. Cela fournit une surface analytique (non-orientable) $M_j$ munie d'une application analytique $\pi_j\colon M_j\to U_j$ qui est un difféomorphisme en dehors du diviseur exceptionnel $\pi_j^{-1}(x_j)$. 
En éclatant $S$ simultanément en tous les $x_j$, on obtient une surface analytique $S'\subset \R^{N'}$ munie d'une application birationnelle propre $\pi$ vers la surface $S$.
La transformée stricte de $\xi$ par $\pi$ est une courbe lisse $\xi'$ de $S'$ transverse au diviseur exceptionnel $\delta=\bigcup_j {\delta_j}$.

\emph{3/ Approximation analytique d'une courbe lisse dans la surface éclatée et implosion.}
Choisissons une paramétrization $u\colon \R/\Z\to \R^{N'}$ de la courbe $\xi'$. Les sommes partielles des séries de Fourier de ses composantes l'approchent pour la topologie $\mathscr{C}^1$ : soit $v \colon \mathbb{S}^1 \to \R^{N'}$ une telle approximation. En paramétrant $\mathbb{S}^1=\{x^2+y^2=1\}$ par les fonctions trigonométriques circulaires et en écrivant $\cos(m\theta)$ et $\sin(m\theta)$ comme des polynômes (de Tchebychev) en ces fonctions, on voit que $v$ est analytique. 
La projection orthogonale sur la surface $S'$ est bien définie et analytique dans un petit voisinage de $S'$ (voir \cite{Kollar:2017}), donc l'image de $v$ définit une courbe localement analytique $\gamma'$ de $S$.
Cette courbe peut être choisie arbitrairement proche de $\xi'$ pour la topologie $\mathscr{C}^1$ et en particulier possédant les mêmes intersections (toutes transverses) avec le diviseur exceptionnel.
La courbe localement analytique est en fait globalement analytique sur $S'$ donc $\gamma=\pi(\gamma')$ est une courbe analytique de $S$ qui répond au problème.
\end{proof}

\begin{rem}[Petite amélioration]
On prouve en fait qu'on peut préserver un nombre arbitraire de dérivées des branches aux singularités : il suffit d'éclater suffisament les singularités.
\end{rem}



\subsection{\'Enumération des courbes combinatoires analytiques de la sphère}

\paragraph{Objectif.} Nous pouvons désormais énumérer les courbes combinatoires analytiques de la sphère, en fonction du nombre de sommets et de leurs degrés.
En effet, la description de la topologie locale d'une singularité analytiques nous a permi de dénombrer les diagrammes de cordes associés.
Comme la topologie globale d'une courbe analytiques en dehors de ses singularités ne possède pas d'obstructions, il ne reste plus qu'à compter ces configurations globales, ce que Tutte a déjà fait pour nous dans \cite{Tutte:1962}. Les objets combinatoires de cette section (courbes, cartes, découpages, dissections) seront supposés connexes.

\paragraph{Découpages de Tutte.} Soit $P$ une sphère privée de $s$ disques. L'orientation de la sphère en induit une sur le bord. Nous dessinons dans le plan qu'il faudra imaginer compactifié par un point à l'infini, avec l'orientation trigonométrique. Les bords sont donc orientés en sens horaire.
Indexons les composantes de bords $J_1,\dots,J_s$ et disposons sur $J_l$ un nombre pair non nul $2k_l$ de points, en quantité totale $2c$. Notons $P_k$ une surface ainsi décorée, où $k$ dénote un multi-indice $(k_1,\dots,k_s)\in \left( \N^*\right)^s$.
Nous appelons \emph{découpage} de $P_k$ toute carte combinatoire connexe ayant pour sommets l'ensemble des points sur son bord, et pour arêtes les arcs du bord entre deux points ainsi qu'un ensemble de courbes simples et disjointes reliant tous les points deux à deux.
Disons que le découpage est \emph{marqué} si nous avons distingué un sommet par composante de bord.
Deux découpages marqués sont égaux s'ils sont images l'un de l'autre par un difféomorphisme de la surface préservant l'orientation et les sommets distingués. Cela revient à considérer ces cartes modulo isotopies relativement au bord de l'espace ambiant.
\begin{figure}[H]
 \centering
  \includegraphics[width=0.55\textwidth]{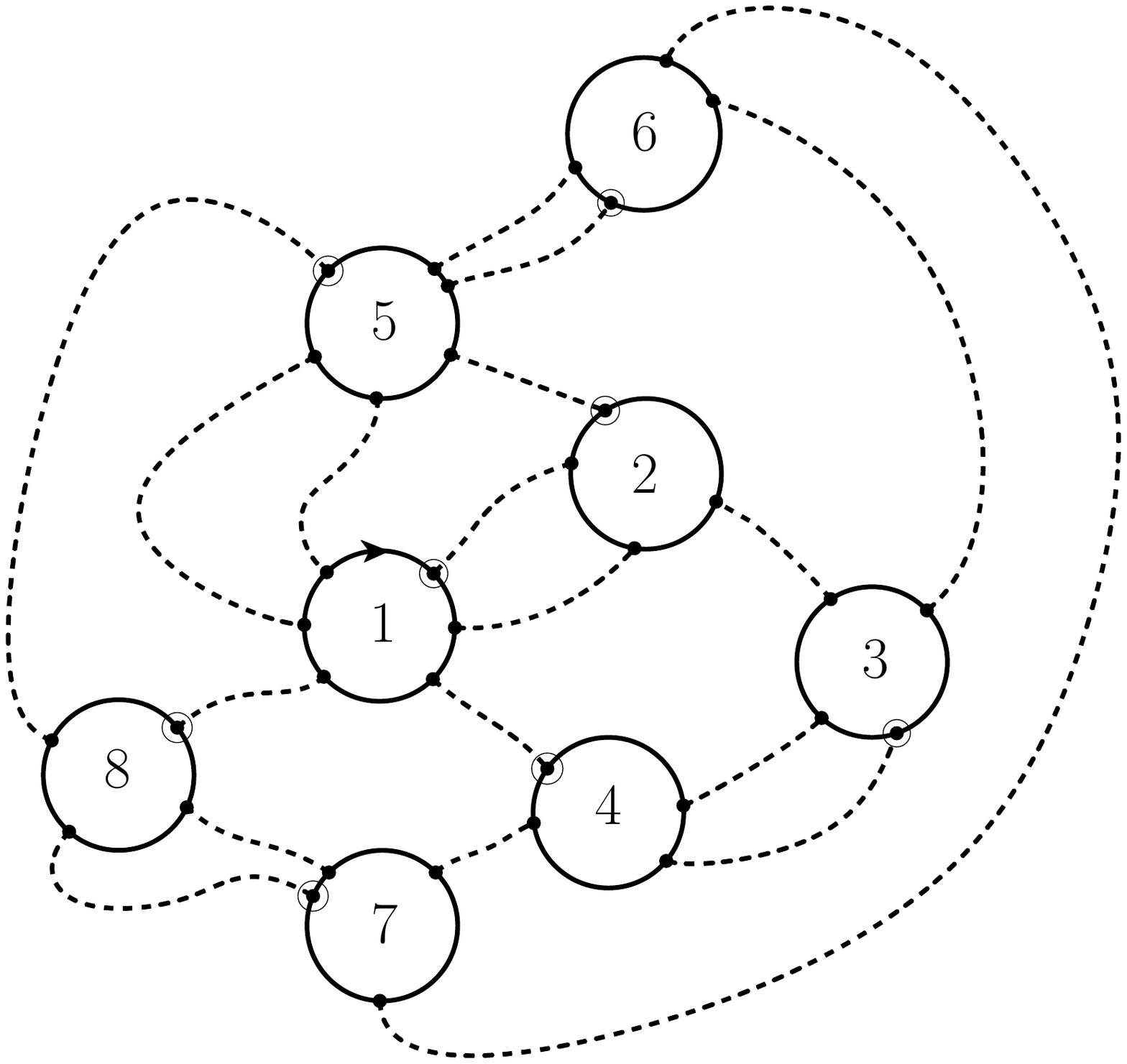}
 \caption{\label{fig:decoupage} Un découpage de $P_{(6,4,4,4,6,4,4,4)}$.}
\end{figure}
Tutte a calculé dans \cite{Tutte:1962} le nombre de découpages marqués de $P_k$ pour s>0, avec la convention $n!=1$ si $n<0$:
\[
\frac{(c-1)!}{(c-s-2)!} \: \prod_{v=1}^{s}{k_v \binom{2k_v}{k_v}}    
.\]

\begin{Prop}
Le nombre de courbes combinatoires analytiques de la sphère ayant $c$ arêtes pour $s$ sommets indexés et marqués de tailles respectives $k_1,\dots, k_v$ vaut:
\[
\frac{(c-1)!}{(c-s-2)!} \: \prod_{v=1}^{s}{k_v \binom{2k_v}{k_v}A_{k_v}}
.\]
\end{Prop}

\paragraph{Dissections.}
L'indexation des composantes de bord et le choix d'un sommet distingué sur chacune d'elles représentent beaucoup de décorations. Montrons comment retrouver tout cela avec l'enracinement d'un seul sommet. Appelons donc \emph{dissection} la donnée analogue à celle d'un découpage, où l'on a oublié la numérotation des composantes de bords ainsi que tous les sommets enracinés sauf un seul: sa racine. Deux dissections sont égales si elles sont images l'une de l'autre par un homéomorphisme préservant l'orientation et la racine.

\paragraph{De la dissection au découpage: numérotation des composantes de bord.}
La dissection induit une carte enracinée de la sphère (sous-entendu connexe), obtenue en contractant chaque composante de bord en un point. Les sommets de cette carte sont donc les composantes de bords, les arêtes sont données par les segments entre deux sommets de la dissection autres que les arcs du bord. Il peut y avoir des boucles et des arêtes multiples, comme d'habitude avec ce genre d'objets. Les demi-arêtes de la carte correspondent aux sommets de la dissection, et les orientations cycliques aux sommets sont données par celle des composantes de bords. Le sommet enraciné de la dissection enracine la carte en une demi-arête, toujours en accord avec l'usage.
Commençons par définir une numérotation privilégiée des composantes de bords grâce à cette carte enracinée de la sphère.

Numérotons par $1$ la composante de bord contenant le sommet racine. Ensuite effectuons un parcours en largeur (BFS) du graphe sous-jacent à la carte en partant de la racine et selon l'ordre suivant. On impose de visiter les voisins (non précédemment visités) d'un sommet $x$ dans l'ordre induit par l'orientation des demi-arêtes partant de $x$. Quand tous les voisins de $x$ ont été visités, on reprend le parcours au plus petit sommet déjà numéroté dont tous les voisins n'ont pas été visités. Ce procédé termine, et par connexité il visite tous les sommets de la carte un par un, induisant un ordre sur les composantes de bords. La figure \ref{fig:dissection} montre la numérotation des composantes de bords ainsi obtenue par l'enracinement au sommet rouge. On visite en premier les bords voisins $2,3,4,5$ du bord $1$, puis on se place en $2$, on visite $6$, puis on se place en $3$, on visite $7$, puis on se place en $4$ et on visite $8$. Ce paragraphe montre qu'une dissection possède un \emph{passeport} bien défini $k=(k_1,\dots,k_n)$ où $2k_j$ est nombre de points sur le bord numéro $j$.

\paragraph{De la dissection au découpage: marquage.}
Montrons désormais comment cette numérotation des composantes de bords avec un sommet enraciné sur $J_1$, permet de faire un choix de sommet distingué sur chaque autre composante de bord $J_l$. A ces fins introduisons une relation d'ordre totale sur l'ensemble $\Cc_l$ des chemins de la carte qui partent de la racine, aboutissent en un point sur $J_l$, rencontrent chaque composante de bord au plus une fois, et parcourent les composantes de bord selon l'orientation induite par celle de la surface.
Par connexité, cet ensemble est non vide, et l'extrémité du plus petit chemin sera le point distingué recherché sur $J_l$.

Un chemin $p \in \Cc_l$ recontre une suite de composantes de bords, que l'on complète à gauche par des zéros pour que sa longueur soit $s$; on la note $u_p$. Par exemple, il y a quatre chemins dans $\Cc_8$ ayant pour suite $(0,0,0,0,0,1,5,8)$ sur la figure \ref{fig:decoupage}, indépendamment de l'emplacement de la racine sur $J_1$.
Ensuite, un élément $c\in \Cc_l$ traverse un certain nombre d'arêtes dans chaque composante de bord, notons $v_p$ la suite de ces entiers. On dit alors que $p<p'$ si $u_p<u_{p'}$ pour l'ordre lexicographique, ou si $u_p=u_{p'}$ et $v_p<v_{p'}$ pour l'ordre lexicographique.
Il est clair que $u_p$ et $v_p$ déterminent $p$ donc c'est bien une relation d'ordre. Il est par ailleurs linéaire par construction, donc $\Cc_l$ admet un plus petit élément et l'on choisit son extrémité comme point distingué de $J_l$.
Les sommets encerclés de la figure \ref{fig:dissection} représentent les sommets marqués pour le choix de la racine rouge sur $J_1$.
\begin{figure}[H]
 \centering
  \includegraphics[width=0.55\textwidth]{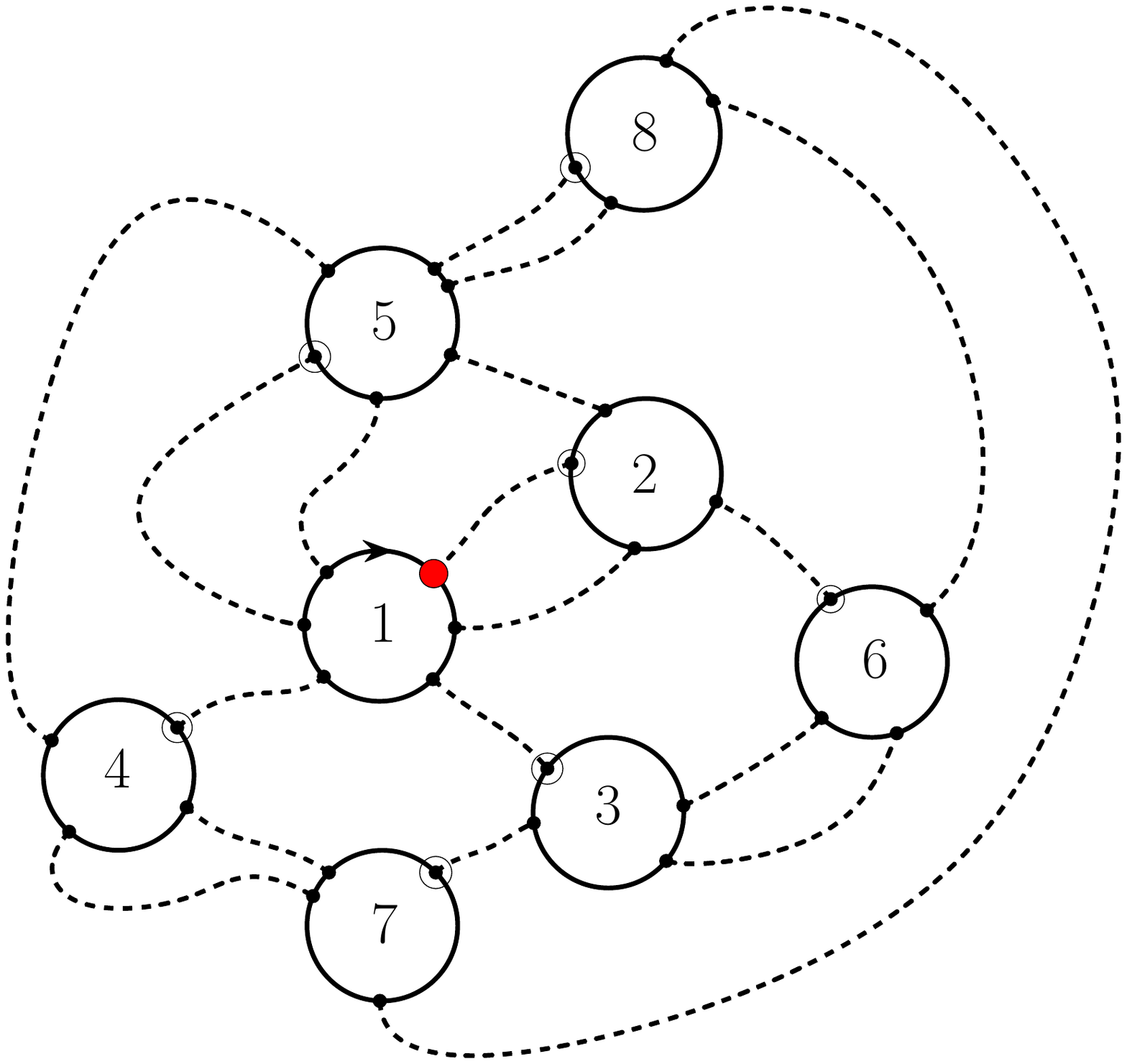}
 \caption{\label{fig:dissection} Dissection enracinée au point rouge et découpage marqué induit.}
\end{figure}

\paragraph{Enumération des courbes combinatoires analytiques enracinées sur la sphère.}

Résumons: l'oubli de l'indexation et du marquage du bord induit une application des découpages marqués vers les dissections enracinées, dont les fibres sont de cardinal $(s-1)!\prod_{v>1}{2k_v}$.
Par ailleurs, une dissection possède un ordre privilégié sur ses composantes de bords ainsi qu'un sommet distingué sur chacune d'entre-elles. Son passeport est la suite $(k_1,\dots, k_n)$ des tailles de ses bords.

Une courbe combinatoire est  enracinée si la carte combinatoire sous-jacente est enracinée par le choix d'une demi-arête. Dans notre contexte cela équivaut à choisir un élément distingué de l'ensemble $R_1$. Les paragraphes précédents permettent d'en déduire un ordre total sur les ensembles des rayons $R_l$, et de choisir un rayon distingué dans chacun d'eux.

Ainsi, étant donné une dissection avec une suite de diagrammes enracinés ayant les bonnes tailles, il y a une façon canonique de les recoller sur les composantes de bord, de manière à récupérer une courbe combinatoire enracinée avec les bonnes tailles.
Réciproquement, une courbe combinatoire analytique enracinée de la sphère se décompose, d'une part en une suite de $C_1,\dots,C_s$ diagrammes analytiques linéaires ayant $k_1,\dots,k_s$ cordes, et d'autre part en un découpage de $P_k$. La donnée de cette décomposition caractérise la courbe de manière unique. Ceci prouve la proposition qui suit.

Insistons sur le fait que nous ne disons rien sur le nombre de courbes enracinées dont les degrés des sommets appartiennent à un ensemble donné, même en isolant la taille de la racine à part: le passeport est un uplet. Cette dernière question semble bien plus difficile.

\begin{Prop}\label{nombre_courbes_global}
Le nombre de courbes combinatoires analytiques enracinées de passeport $k$ sur la sphère vaut:
\[
2k_1\, \frac{(c-1)!}{(s-1)!(c-s-2)!} \: \prod_{v=1}^{s}{\binom{2k_v}{k_v}\frac{A_{k_v}}{2}}
.\]
\end{Prop}

\begin{rem}[Autres surfaces]
Tout le travail effectué jusqu'à présent se généralise aux surfaces de genre supérieur. La description de la topologie des courbes combinatoires analytiques et la méthode d'enracinement (qui n'utilise pas la planarité) permettent de décomposer l'énumération en une partie locale et une autre globale. La première correspond aux diagrammes analytiques linéaires, et l'autre à un découpage enraciné d'une surface à bord, avec un passeport $k$ donné. Si $m_k(g)$ dénombre ces derniers, alors le nombre de courbes combinatoires enracinées analytiques sur la surface compacte de genre $g$ vaut:
$m_k(g) \prod_{v=1}^{s}{A_{k_v}}$
\end{rem}

\subsection{Majoration en fonction des paramètres topologiques}

Nous majorons désormais le nombre de courbes combinatoires analytiques de la sphère par une fonction exponentielle du nombre d'arêtes indépendament de la croissance du nombre de singularités et de leurs degrés.

\begin{Prop}\label{maj_cc_sphere}
Le nombre $Calc_{\, \S^2 \,}(c)$ de courbes combinatoires analytiques enracinées de la sphère ayant $c$ arêtes vérifie, pour une certaine constante $\rho$ l'inégalité: $Calc_{\, \S^2 \,}(c) \leq c^3\, \rho^c$.
On peut choisir $\rho\leq 96\,e^\frac{1}{3}<134$ donc $Calc_{\, \S^2 \,}(c)= o\left(134^c\right)$. 
\end{Prop}

\begin{conj}
D'après le théorème \ref{diaganalin}, si $k>0$ on a l'inégalité $A_{k}\leq a_0'\,k^{-\frac{3}{2}}\,\alpha^{-k}$, pour un certain $a'_0$ indépendant de $k$.
On conjecture, au vu des premiers termes de la suite $A_k$ que $a_0'=\alpha$ convient.
En reprenant la preuve ci-dessous, on montrerait alors qu'on peut choisir $\rho < 83$.
Par ailleurs, une connaissance de la distribution de $s$ et de $k$ quand $c$ tend vers l'infini permettrait d'améliorer encore ce choix.
Peut-on prendre $\rho=4\alpha^{-1}\,\exp\left(\sqrt{\frac{2\,a_0}{\sqrt{\pi}}}\right)$ ?
\end{conj}

\begin{proof}
Fixons d'abord le nombre de singularités $s<c$ et majorons le nombre $Calc_{\,\S^2\,}(c,s)$ de courbes analytiques combinatoires enracinées à $s$ singularités et $c$ arêtes. D'après ce qui précède, nous cherchons à majorer la somme suivante indexée par les uplets $k$ de $s$ entiers strictement positifs dont la somme vaut $c$ (les compositions de $c$ en $s$ parties):
\begin{equation}\label{majo}
Calc_{\,\S^2\,}(c,s)\leq \sum_{[k]=(s,c)}{
2k_1\frac{(c-1)!}{(s-1)!(c-s-2)!} \prod_{v=1}^{s}{\binom{2k_v}{k_v} \frac{A_{k_v}}{2}}}
.\end{equation}
On sait d'après \cite[§1.8]{GhySim:2020} que le nombre de diagrammes analytiques à $k$ cordes est plus petit que $6^{k-1}$ fois le $k-1$ ième nombre de Catalan. Avec $2k$ choix pour l'enracinement on a donc $A_k\leq 2\times 6^{k-1} \binom{2k-2}{k-1}$. Or par récurrence, pour tout $k\in \N$: $\binom{2k}{k}\leq \frac{4^{k}}{\sqrt{3\,k_v+1}}$. Par conséquent:
\[
k_1\prod_{v=1}^{s}{\binom{2k_v}{k_v} \frac{A_{k_v}}{2}} \leq
k_1 \prod_{v=1}^{s}{
\frac{4^{k_v}}{\sqrt{3k_v+1}}
\frac{(6\times4)^{k_v}}{(6\times4)\sqrt{3k_v-2}} }
=
\frac{96^c}{36^s} \prod_{v=1}^{s}{\frac{k_1}{2\sqrt{(k_v+\frac{1}{3})(k_v-\frac{2}{3})}}}\leq \frac{96^c}{36^s}
.\]
En réinjectant dans l'expression \ref{majo}, on reconnaît le nombre de compositions de $c$ en $s$ parties, qui vaut $\binom{c-1}{s-1}=\frac{s}{c}\binom{c}{s}$.
Ensuite on majore $\frac{(c-1)!}{(s-1)!(c-s-2)!}\leq s\,c\, \binom{c}{s}$
et $\binom{n}{p}\leq \left(\frac{e\,n}{p} \right)^p$ pour obtenir:
\[
Calc_{\,\S^2\,}(c,s)\leq 
2\, \frac{(c-1)!}{(s-1)!(c-s-2)!}
\sum_{[k]=(s,c)}{\frac{96^c}{36^s}} \leq
2\, \frac{96^c}{36^s}\,
s\,c\,\binom{c}{s}\,
\frac{s}{c}\binom{c}{s}
\leq 2s^2\,96^c
\left( \frac{e\,c}{6s} \right)^{2s}
.\]
Pour tout $c$, le maximum selon  $1\leq s\leq c$ de $\left(\frac{e\,c}{6s} \right)^{2s}$ est atteint en $s=\frac{c}{6}$, et vaut $e^{\frac{c}{3}}$. Ainsi:

\[
Cacl_{\,\S^2\,}(c,s)\leq
2\,s^2\,\left(96 e^{\frac{1}{3}} \right)^c
.\]

La somme des $2\,s^2$ pour $1\leq s\leq c$ est plus petite que $c^3$ pour $c>3$, ce qui donne la formule voulue dans ce cas; et l'expression annoncée majore le nombre total de courbes combinatoires pour $c\leq 3$, ce qui conclut.
\end{proof}

\begin{quest}[Non-optimalité de la borne]
La majoration de la somme $\sum_{[k]=(s,c)}{ \prod_{v=1}^{s}{ k_v^{-1}}}$ était grossière. A cause de cela, la preuve fonctionne aussi pour les courbes marquées. On devrait pouvoir l'améliorer en partitionnant les compositions en indice selon la taille du carré maximal de la partition qu'elles induisent. Une autre approche serait d'utiliser le fait que c'est le coefficient de $x^c$ de la série $f(x)^s$ où $f$ est la primitive nulle à l'origine de $\frac{1}{x}\,\log \frac{1}{1-x}$ et se ramener à l'étude des singularités de $f$ dans le plan complexe.
\end{quest}

\begin{rem}[Courbes du plan]
La différence entre une carte (respectivement une courbe) combinatoire du plan et de la sphère, revient au choix sur la sphère, d'une face infinie. Il y a $c-s+2$ faces, donc $Calc_{\,\R^2\,}(c,s)\leq (c-s+2)\, Calc_{\,\S^2\,}(c,s)\leq c^4 \rho^c$.
\end{rem}

\section{Courbes algébriques réelles}

Dans cette section nous rappelons d'abord quelques éléments de géométrie algébrique réelle, la source principale étant \cite{BoCoRo:1992}.
Ensuite nous décrivons quelles courbes combinatoires sur une surface munie d'une structure algébrique réelle lisse proviennent de courbes algébriques.
Enfin, nous majorons en fonction du degré le nombre de types topologiques de courbes algébriques projectives réelles.

\subsection{Homologie d'une surface algébrique réelle}

\subsubsection{Quelques précisions sur la notion de courbe algébrique réelle}

Nous nous intéressons à la topologie du lieu réel d'une courbe algébrique singulière sur une surface algébrique projective lisse définie sur le corps des réels.
Ces termes sont employés au sens de la géométrie algébrique classique : les objets sont définis par des équations homogènes en un certain nombre de variables ; les morphismes correspondent à des changements de variables et les applications rationnelles à des changements de variables bien définis en dehors de sous-variétés algébriques de codimension positive. Ainsi, une variété algébrique $X$ définie sur le corps des réels possède donc un lieu réel, parfois noté $X(\R)$, et un lieu complexe noté $X(\C)$ ; de la même façon, un morphisme $f\colon Y\to X$ induit une application partout définie $f(\R) \colon Y(\R)\to X(\R)$ entre les lieux réels (comme nous travaillerons presque toujours avec le lieu réel nous n'aurons pas à spécifier le corps sous-jacent).
En particulier dans une surface, une courbe algébrique, est fermée pour la topologie de Zariski, et son lieu réel n'est pas forcément connexe ni purement de dimension $1$.

Par exemple, si on intersecte le parapluie de Whitney, qui est une surface de $\R^3$ d'équation $x^2z=y^2$, avec une petite sphère centrée autour de sa singularité à l'origine, on obtient une courbe algébrique singulière de la sphère qui possède une composante connexe homéomorphe à un "huit", et une autre qui est un point isolé.
On peut également construire des exemples avec plusieurs composantes de dimension $0$ et $1$.
Les points isolés d'une courbe algébrique correspondent à des branches conjuguées du lieu complexe qui intersectent le lieu réel de la surface en un point.
Il se peut également que deux telles branches complexes conjuguées intersectent le lieu réel de la surface en un point qui est également sur une branche du lieu réel, dans ce cas la courbe présente une singularité en ce point, même si elle n'est pas apparente dans le lieu réel. 

Dans une surface algébrique réelle $S$, une courbe algébrique dont le lieu réel est connexe définit une courbe combinatoire (au sens d'un graphe plongé dans la surface $S$ dont les sommets sont décorés par des diagrammes de cordes, les arêtes d'un sommet étant mises en bijection avec les rayons du diagramme le décorant) comme suit.
Il y a un diagramme de cordes pour chaque singularité ; cela comprend le cas d'une composante de dimension $0$ qui est un point isolé, dans ce cas le diagramme de cordes est \emph{trivial}, il n'a pas de cordes, c'est le mot vide ; cela comprend également le cas des singularités non apparentes, dans ce cas le diagramme de cordes contient une seule corde.
Les singularités sont reliées, comme dans le cas analytique, par les arcs lisses de la courbe.
Une courbe combinatoire sera dite \emph{algébrique} si elle est ainsi associée à une courbe algébrique. Elle sera dite \emph{sans points isolés} si elle ne contient pas de diagrammes de cordes triviaux (cela équivaut au fait que le graphe ne possède pas de sommets isolés duquel ne sort aucune arête).

Le théorème principal de cette section \ref{cocalg}, celui de réalisabilité d'une courbe combinatoire par une courbe algébrique, concernera les courbes combinatoires sans points isolés, autrement dit le graphe pourra ou non être connexe, il pourra également posséder des sommets isolés, mais alors ceux-ci ne seront pas décorés par des diagrammes sans cordes.


\subsubsection{Homologie algébrique d'une surface algébrique réelle}

Soit $S$ une surface projective réelle \textit{lisse}.
Afin de motiver l'introduction des éléments de géométrie algébrique dans ce paragraphe, rappelons le schéma de la preuve du théorème \ref{cocan}, car nous l'adapterons pour réaliser une courbe combinatoire $\Gamma$ par une courbe algébrique de $S$. Pour le problème de sa réalisation analytique, nous avons d'abord éclaté afin de nous ramener au monde lisse, où nous avons pu y appliquer des résultats d'approximation résultant des travaux de Whitney, Grauert et Cartan. En particulier, le fait qu'une courbe localement analytique soit globalement le lieu d'annulation d'une fonction analytique correspond à l'annulation d'un certain groupe de cohomologie (le premier groupe de cohomologie du faisceau des fonctions analytiques, c'est d'ailleurs comme cela que Cartan l'obtient \cite[no. 7]{Cartan:1957}).

Pour introduire le résultat d'approximation analogue dont nous aurons besoin dans le cas algébrique, définissons d'abord le premier groupe d'homologie algébrique de $S$.
Une courbe algébrique singulière se décompose en un complexe cellulaire de dimension $1$ en prenant chaque point singulier pour une $0$-cellule et en ajoutant une $0$-cellule par composante lisse, le reste forme des arcs lisses entre ces points. Ce complexe cellulaire est un cycle modulo $2$ et définit donc un élément de $H_1(S;\Z/2\Z)$ appelé la \emph{classe fondamentale} de la courbe algébrique (voir \cite[11.3.2]{BoCoRo:1992}).
%
%
%
%
Notons $H_1^{alg}(S;\Z/2\Z)$ le sous espace vectoriel engendré par les classes algébriques. On définit la cohomologie algébrique par dualité de Poincaré.

Répétons la définition \cite[12.4.10]{BoCoRo:1992} : une sous-variété $Y$ compacte et $\CC^\infty$ d'une variété algébrique réelle non singulière et affine $X$ \emph{possède une approximation algébrique} dans $X$ si pour tout voisinage $\Omega$ dans $\CC^\infty(Y,X)$ de l'application d'inclusion $i\colon Y \to X$ il existe $h\in \Omega$ tel que $h(Y)$ est un fermé de Zariski non singulier de $X$.
Nous pouvons maintenant énoncer le théorème \cite[12.4.11]{BoCoRo:1992} de Nash-Tognoli dans le cas où la variété $X$ est de dimension $2$ : une hypersurface $Y$ compacte et lisse de $X$ possède une approximation algébrique si et seulement si la classe d'homologie $[Y]$ définie par $Y$ est algébrique, c'est-à-dire appartient à $H_1^{alg}(X;\Z/2\Z)$.

Pour satisfaire à l'hypothèse de non-singularité, nous appliquerons ce théorème dans un éclaté $X$ de la surface $S$ à la transformée stricte $Y$ d'une sous-variété singulière $C$ de $S$. Nous devons donc nous assurer du bon comportement des sous-groupes des classes d'homologie algébriques par éclatement.
%
%
%
Soit $p\colon S' \to S$ l'éclatement de $S$ en un point et $E\subset S$ le diviseur exceptionel. On a $H_1(S';\Z/2\Z)=H_1(S;\Z/2\Z)+ \Z/2\Z\cdot [E]$. Comme $[E]$ est clairement algébrique et comme toute classe algébrique de $S$ se relève en une classe algébrique de $S'$ en prenant sa transformée stricte, on a que $H_1^{alg}(S;\Z/2\Z)+\Z/2\Z\cdot[E] = H_1^{alg}(S';\Z/2\Z)$ (par abus de notation, le groupe d'homologie dans le membre de gauche fait ici référence au dual de Poincaré de l'image réciproque par $p^*$ du premier groupe de cohomologie algébrique de $S$, voir \cite{BoCoRo:1992} pour plus de détails).

Ainsi, une sous-variété analytique $C$ de $S$ définit une classe d'homologie algébrique si et seulement si c'est le cas de sa transformée stricte dans $S'$, et une sous-variété $Y$ de $S'$ définit une classe d'homologie algébrique si et seulement si c'est le cas de sa projection par $p$ dans $S$.
%
%
La discussion est la même si $p\colon S' \to S$ est l'éclatement itéré en un nombre fini de points.
Ceci montre également que toute classe d'homologie d'une surface rationnelle (c'est-à-dire birationnelle au plan projectif, donc en particulier la sphère) est algébrique.

\subsection{Courbes combinatoires algébriques}

Répétons par souci de clarté les éléments utiles pour la formulation du résultat de cette section.
Comme remarqué pour le théorème \ref{cocan}, la notion de courbe combinatoire dans le théorème de ce paragraphe peut être prise au sens plus large d'un graphe de diagrammes de cordes plongé dans la surface (c'est l'idée que l'on se fait d'un tracé topologique d'une courbe singulière sur la surface) ; autrement dit il n'est pas nécessaire que la surface $S$ dans laquelle le graphe est plongé soit celle associée à la carte combinatoire sous-jacente à la courbe combinatoire (ce qui est le cas si et seulement si la courbe combinatoire remplit la surface, c'est-à-dire que son complémentaire soit une union de disques, en particulier le graphe doit être connexe).
Une courbe combinatoire de $S$ est \emph{algébrique et sans points isolés} si elle est associée à une courbe algébrique réelle sur la surface n'ayant pas de points isolés.

Rappellons qu'un diagramme de corde provient d'une singularité de courbe algébrique si et seulement s'il est analytique ; et qu'on dit d'une courbe combinatoire qu'elle vérifie l'hypothèse topologique locale si ses sommets sont décorés par des diagrammes de cordes analytiques.
Nous avions introduit les brins d'une courbe combinatoire dans la section \ref{formulation_combinatoire}. Chaque brin définit une classe d'homologie dans $H_1(S;\Z/2\Z)$ et on appelle la somme de ces classes sur tous les brins la \emph{classe d'homologie de la courbe}.
\begin{Define}
Nous dirons qu'une courbe combinatoire vérifie \emph{l'hypothèse globale} si sa classe d'homologie est algébrique, c'est-à-dire qu'elle apartient à $H_1^{alg}(S,\Z/2\Z)$.
\end{Define}

\begin{Thm} \label{cocalg}
Soit $S$ une surface algébrique. Toute courbe combinatoire sans points isolés de $S$ vérifiant les hypothèses locale et globale est algébrique.
En particulier si $S$ est rationnelle, par exemple la sphère ou le plan projectif, toute courbe combinatoire non triviale vérifiant l'hypothèse locale est algébrique.
%
\end{Thm}

\begin{proof} Bien sûr une courbe combinatoire algébrique vérifie les hypothèses locales et globales. Réciproquement soit $\Gamma$ une courbe combinatoire sur une surface algébrique $S$ vérifiant les hypothèses locales et globales. Le schéma de la preuve est essentiellement le même que pour le théorème \ref{cocan}, on utilise les notions discutées dans la section précédente.

\emph{1/ Préparation du terrain.}
On commence par appliquer le théorème \ref{cocan} pour réaliser $\Gamma$ par une courbe analytique, que l'on peut paramétrer par $i\colon M \to S$ avec $M$ une union de cercles.
Comme remarqué après ce théorème, on peut fixer pour $k$ fini quelconque, les $k$-jets de $i$ aux singularités (pourvu que la topologie requise par les diagrammes de cordes soit respectée).
%

\emph{2/ Eclatement et approximation de la transformée stricte.}
Notons $p \colon S'\to S$ un éclatement itéré de $S$ résolvant les singularités de $i$.
Soit $i'\colon M\to S'$ la transformée stricte de $i$.
Comme la classe de $i$ dans $H_1(S;\Z/2\Z)$ est algébrique, celle de $i'$ dans $H_1(S;\Z/2\Z)$ l'est aussi.
Mais $i'$ est également lisse, donc par le théorème \cite[Thm 12.4.11]{BoCoRo:1992} de Nash-Tognoli on peut approcher $i'$ par une courbe algébrique $j'$ arbitrairement près pour la topologie $\CC^\infty$. Notons $J'$ l'image de $j'$.

\emph{3/ Implosion.}
La projection $J=p(J')$ est une courbe algébrique de $S$ dont les composantes de dimension $1$ approchent $i$ arbitrairement près pour la topologie $\CC^\infty$, avec les mêmes $k$-jets aux singularités (l'égalité des $k$-jets se déduit de l'approximation pour la topologie $\CC^0$ de $i$ pourvu qu'on ait itéré de manière appropriée les éclatements).
Comme $p$ est un homéomorphisme en dehors des singularités, il ne peut y avoir de composantes "supplémentaires" de dimension $0$ dans $J$ qu'à ses points singuliers.
La courbe $J$ répond donc au problème.
%
%
\end{proof}

\begin{quest}
Peut-on réaliser les courbes combinatoires sans introduire d'avantage de multiplicité aux singularités provenant des branches conjugués du lieu complexe ?
\end{quest}

\begin{quest}
Etant donné, dans une surface topologique, une courbe combinatoire non triviale vérifiant l'hypothèse locale : existe-t-il un modèle algébrique de la surface tel que la courbe combinatoire soit réalisée par une courbe algébrique ?
La réponse serait positive si étant donné dans une surface topologique, une classe du premier groupe d'homologie, on savait qu'il existe un modèle algbrique de la surface tel que cette classe soit algébrique. Le théorème \cite[11.3.8]{BoCoRo:1992} s'en rapproche mais n'est pas tout a fait ce dont nous avons besoin.
\end{quest}

\subsection{Majoration du nombre de courbes algébriques par le degré}

Nous cherchons désormais à majorer le nombre de types topologiques de courbes algébriques réelles singulières de degré $d$ du plan projectif dont le lieu réel est connexe. Un \emph{type topologique} désigne une classe d'isotopie dans la catégorie des courbes lisses par morceaux, enrichi par les informations locales données par les halos réels des singularités. C'est précisément la donnée de la courbe combinatoire.
Comme toute classe d'homologie du plan projectif est algébrique, nous cherchons donc à majorer le nombre de courbes combinatoires connexes vérifiant l'hypothèse locale.
Nous commençons par les majorer en fonction du nombre d'arêtes pour aboutir ensuite des fonctions du degré via les formules de Plücker généralisées.

\begin{rem}
Dans \cite{KarlaOrev:2003, KarlaOrev:2004}, Kharlamov et Orevkov encadrent asymptotiquement le nombre de classes d'isotopie de courbes algébriques lisses (pas nécessairement connexes dans le lieu réel) de degré $d$ du plan projectif réel, par des expressions de la forme $\exp{(Cd^2+o(d^2))}$.
\end{rem}

La généralisation de Rosenlicht des formules de Plücker (voir par exemple \cite{Dieu:1974}), entraîne que toute courbe algébrique du plan projectif réel vérifie:

\[\sum_{v,\; k_v=1}{1}+\sum_{v,\; k_v>1}{\frac{k_v(k_v-1)}{2}}
\leq \frac{(d-1)(d-2)}{2} \mathrm{\quad et\; donc\quad } c=\sum_{v}{k_v}\leq \frac{(d-1)(d-2)}{2}.\]
Avec la borne précédente selon le nombre d'arêtes nous en déduisons le théorème suivant.

\begin{Thm}
Le nombre $Cal_{\,\R\P^2\,}(d)$ de courbe combinatoires enracinées algébriques de degré $d$ du plan projectif vérifie, pour une certaine constante $\rho$, la majoration:
\[Cal_{\,\R\P^2\,}(d)
\leq \frac{d^8}{16} \rho^\frac{d^2}{2}
.\]
Comme on peut prendre $\rho < 134$, on a $Cal_{\,\R\P^2\,}(d)=o\left(12^{d^2}\right)$.
\end{Thm}

\begin{rem}
Si comme conjecturé, $\rho = 83$ convenait; alors on aurait $Cal_{\,\R\P^2\,}(d)=o\left(10^{d^2}\right)$.
\end{rem}

\begin{proof}
Considérons une courbe combinatoire $\Gamma$ provenant d'une courbe algébrique $\gamma$ de degré $d$ du plan projectif. Enracinons $\Gamma$ en une demi-arête $r$. Géométriquement, on peut penser à $r$ comme à un segment issu d'un point singulier $r_-$ de $\gamma$, dont l'autre extrémité est un point $r_+$ à mi-chemin d'une autre singularité. Notons $s$, $c$ et $k$ respectivement le nombre d'arêtes, de sommets, et le passeport (degrés des sommets), de $\Gamma$.

Choisissons une droite $D$ qui intersecte transversalement la courbe sans passer par ses singularités, et supposons qu'elle passe par $r_+$.
Contractons désormais la droite $D$; alors notant avec un prime les images par la projection: on obtient une courbe algébrique $\gamma'$ dans la sphère ayant une singularité de plus en $r'_+$, dont le diagramme associé est du type $T$ avec $k_{s+1}\leq d$ cordes. Enracinons sa courbe combinatoire $\Gamma'$ en la demi-arête issue de $r'_-$ empruntant l'arête $r'$. La courbe combinatoire $\Gamma'$ possède $c'\leq c+d$ arêtes. Les arêtes supplémentaires proviennent de la duplication de celles de $\Gamma$ qui intersectaient la droite $D$.

Faisons pour chaque courbe combinatoire $\Gamma$ provenant d'une courbe algébrique de degré inférieur ou égal à $d$, le choix d'une telle droite $D$. On obtient donc une application des courbes combinatoires enracinées algébriques de degré plus petit que $d$ du plan projectif vers les courbes combinatoires enracinées à moins de $\frac{(d-1)(d-2)}{2}+d\leq \frac{d^2}{2}$ arêtes dans la sphère.
Majorons le cardinal de ses fibres en énumérant le nombre maximum de relevés distincts d'une courbe combinatoire de la sphère dans l'image de cette application. Pour relever une courbe $\Gamma'$ de la sphère ayant (au plus) $s+1$ sommets et $c+d$ arêtes, il suffit de choisir un sommet dont le diagramme de cordes est du type $T$, et l'éclater. Il y a au plus $s+1\leq c+d \leq \frac{d^2}{2}$ tels diagrammes.
Par conséquent, en utilisant la proposition \ref{maj_cc_sphere}:
\[Cal_{\,\R\P^2\,}(d)
\leq \frac{d^2}{2}\, Calc_{\,\S^2\,}\left(\frac{d^2}{2}\right)\leq
\left(\frac{d^2}{2}\right)^4 \rho^\frac{d^2}{2}.\]
\end{proof}

\bibliographystyle{alpha}
\bibliography{main.bib}

\end{document}